%% file: slra_algorithm.tex
\documentclass[1p, preprint]{elsarticle}
\usepackage{amsmath,amssymb,mathdots,mathrsfs}
\usepackage{amsthm}
\usepackage{algorithmic}
\usepackage{comment}
\usepackage{color}
\usepackage{times}
\usepackage{url}

\usepackage{subfigure}

\input{comsymb}

\input{comdefs}
\input{slraalgdefs}

\input{engthm}

\def\proofapp{\begin{proof}The proof is given in \ref{sec:proofs}.\end{proof}}

\begin{document}

\begin{frontmatter}



\title{Variable projection for affinely structured \\ low-rank approximation in weighted $2$-norms}
\author[vub]{Konstantin Usevich\corref{cor1}}
\ead{Konstantin.Usevich{@}vub.ac.be}
\author[vub]{Ivan Markovsky}
\ead{Ivan.Markovsky{@}vub.ac.be}

\cortext[cor1]{Corresponding author}

\address[vub]{Vrije Universiteit Brussel, Department ELEC, Pleinlaan 2, B-1050, Brussels, Belgium}

\begin{abstract}
The structured low-rank approximation problem for general affine structures, weighted $2$-norms and fixed elements is considered. The variable projection principle is used to reduce the dimensionality of the optimization problem. Algorithms for evaluation of the cost function, the gradient and an approximation of the Hessian are developed. For $m \times n$ mosaic Hankel matrices the  algorithms have complexity $O(m^2 n)$. 
\end{abstract}


\end{frontmatter}

\section{Introduction}

An \textit{affine matrix structure} is an affine map from a \textit{structure
 parameter space} $\bbR^{\np}$ to a space of matrices $\bbR^{m \times n}$, defined by
\begin{equation}
\struct(p) = S_0 + \sum\limits_{i=1}^{\np} p_k S_k, 
\label{eq:struct_def}\tag{$\struct$}
\end{equation}
where $S_k \in \bbR^{m\times n}$. Without loss of generality, we can consider only the case $m \le n$. The \textit{structured low-rank approximation} is the problem of finding the best
low-rank structure-preserving approximation of a given data matrix 
 \cite{Markovsky08A-Structured,Markovsky12-Low}.

\begin{problem}[Structured low-rank approximation]\label{prob:slra1}
Given an affine structure $\struct$, data vector $\sdat{p} \in \bbR^{\np}$,  norm $\|\cdot\|$  and natural number $r < \min(m,n)$ 
\begin{equation}
\minim_{\sest{p} \in \bbR^{\np}} \| \sdat{p} - \sest{p} \| \;  \sto \;
\rrank \struct(\sest{p}) \le r.
\label{eq:prob_slra}\tag{SLRA}
\end{equation}
\end{problem}
In this paper, we consider the case of a \textit{weighted $2$-norm}, given by 
\begin{equation}
\|p\|^2_\wmat := p^{\top} \wmat p, \quad \wmat \in \bbR^{\np \times \np},
\label{eq:wnorm}\tag{$\|\cdot\|^2_\wmat$}
\end{equation}
where $\wmat$ is
\begin{itemize}
\item either  a symmetric positive definite matrix,
\item or a diagonal matrix
\begin{equation}
\wmat = \diag(w_1,\ldots,w_{\np}),\quad w_i \in (0; \infty],
\label{eq:wtoW}\tag{$w \rightarrow \wmat$}
\end{equation}
 where $\infty\cdot 0 = 0$ by convention. A \textit{finiteness} constraint  $\| \sdat{p} - \sest{p} \|^2_{\wmat} < \infty$ is additionally imposed on \eqref{eq:prob_slra}. Problem \eqref{eq:prob_slra} with the weighted norm \eqref{eq:wnorm} given by \eqref{eq:wtoW} is equivalent to \eqref{eq:prob_slra} with an \textit{element-wise weighted} 2-norm 
\[
\|p\|^2_{w} = \sum_{w_i \neq \infty} w_i p^2_i,
\]
and a set of \textit{fixed values} constraints
\[
(\sdat{p})_i = \sest{p}_i \;\mbox{for all}\; i \;\mbox{with}\; w_i = \infty. 
\]
\end{itemize}
The structured low-rank approximation problem with the weighted $2$-norm appears in signal processing, computer algebra, identification of dynamical systems, and other applications. We refer the reader to \cite{Markovsky08A-Structured,Markovsky12-Low} for an overview. In this paper, we consider general affine structures \eqref{eq:struct_def} and, in particular, structures that have the form
\begin{equation}
\struct(p) = \Phi \HH_{\bfm, \bfn}(p),
\label{eq:mosaic_like}\tag{$\Phi\HH_{\bfm, \bfn}$}
\end{equation}
where $\Phi$ is a full row rank matrix and $\HH_{\bfm, \bfn}$ is a mosaic Hankel \cite{Heinig95LAA-Generalized} matrix structure.

Many data modeling problems can be reduced to \eqref{eq:prob_slra} with the structure \eqref{eq:mosaic_like} and weighted norm, defined by \eqref{eq:wtoW}, see \cite{Markovsky08A-Structured,Markovsky.Usevich13JCAM-Software}. In data modeling, the number of rows $m$ usually has the meaning of the model complexity and the number of columns $n$ is of the same order as the number of data points \cite{Markovsky12-Low}. Typically, the case $m \ll n$ is of interest, i.e. approximation of a large amount of data by a low-complexity model.

\subsection{Mosaic Hankel structure}

A \textit{mosaic Hankel} matrix structure $\HH_{\bfm,\bfn}$ \cite{Heinig95LAA-Generalized} is a map defined by two integer vectors
\begin{equation}
\mathbf{m} = \bmx m_1 & \cdots & m_q \emx \in \bbN^q 
\qquad\text{and}\qquad
\mathbf{n} = \bmx n_1 & \cdots & n_N \emx \in \bbN^N
\label{mn}\tag{$\bfm,\bfn$}
\end{equation}
as follows:
\begin{equation}
\HH_{\bfm, \bfn}(p) \defeq 
\begin{bmatrix} 
\HH_{m_1,n_1}({p}^{(1,1)})     & \cdots & \HH_{m_1,n_N}({p}^{(1,N)}) \\ 
\vdots                    &        & \vdots \\ 
\HH_{m_q,n_1}({p}^{(q,1)})     & \cdots & \HH_{m_q,n_N}({p}^{(q,N)}) 
\end{bmatrix} , 
\label{eq:mosaic}\tag{$\HH_{\bfm, \bfn}$}
\end{equation}
where
\begin{equation}
p = \vcol(p^{(1,1)}, \ldots, p^{(q,1)},\ldots,p^{(1,N)}, \ldots, p^{(q,N)}), 
\qquad
p^{(k,l)} \in \bbR^{m_k+n_l-1},
\label{eq:p}\tag{$p$}
\end{equation}
is the partition of the parameter vector, and $\HH_{m,n}: \bbR^{m+n-1} \to \bbR^{m\times n}$ is the \textit{Hankel} structure
\[
\HH_{m,n}(p) :=
\begin{bmatrix} 
p_1    & p_2     & p_3     & \cdots    & p_n \\
p_2    & p_3     & \iddots &           & p_{n+1} \\
p_3    & \iddots &         &           &  \vdots\\
\vdots &         &         &           & p_{m+n-2}\\
p_m    & p_{m+1} & \cdots  & p_{m+n-2} & p_{m+n-1}
\end{bmatrix}.
\]
Note that the number of parameters for \eqref{eq:mosaic} is equal to
\[
\np = N \sum\limits_{k=1}^q m_k + q \sum\limits_{l=1}^N n_l - Nq,
\]
the number of columns is equal to $\sum\limits_{k=1}^{q} m_k$, and the number of rows is $\sum\limits_{l=1}^{N} n_l$.

The mosaic Hankel  structure is a generalization of the \textit{block-Hankel} structure, which is defined as
\begin{equation}
\HH_{m_1,n_1}(C) \defeq
\begin{bmatrix} 
C_{1,:,:}    & C_{2,:,:}       & \cdots          & C_{n_1,:,:}      \\
C_{2,:,:}    & \iddots   &         & C_{n_1+1,:,:}    \\
\vdots &  \iddots         &                       & \vdots \\
C_{m_1,:,:} & C_{m_1+1,:,:} &  \cdots &   C_{m_1+n_1-1,:,:}
\end{bmatrix},
\label{eq:blh}\tag{$\HH_{m_1,n_1}(C)$}
\end{equation}
where $C \in \bbR^{(m_1+n_1-1)\times q\times N}$ is a $3$-dimensional tensor. Indeed, consider permutation matrices
\begin{equation}
\Pi^{(k,l)} := \bmx \eyemat{k} \otimes \ortvec{1} & \ldots  & \eyemat{k} \otimes \ortvec{l} \emx^\top,
\label{eq:perm}\tag{$\Pi^{(k,l)}$}
\end{equation}
where $\otimes$ is the Kronecker product. Then the block-Hankel structure \eqref{eq:blh} can be transformed to a mosaic Hankel structure \eqref{eq:mosaic} by permutation of the rows and columns
\begin{equation}
\Pi^{(m_1, q)} \HH_{m_1, n_1} (C) (\Pi^{(n_1,N)})^{\top} = \HH_{\bsm m_1 & \cdots & m_1 \esm, \bsm n_1 & \cdots & n_1 \esm} (p),
\label{eq:bh_to_mosaic} \tag{$\HH_{m_1, n_1} (C) \leftrightarrow \HH_{\bfm, \bfn}$}
\end{equation}
where $p$ is defined as \eqref{eq:p} with $p^{(k,l)}_{i} \defeq C_{i,k,l}$ (an unfolding of the tensor $C$).

\subsection{Optimization methods and the variable projection}
Problem \eqref{eq:prob_slra} is non-convex and except for a few special cases (\eg~for unstructured matrices and Frobenius norm, for circulant matrices, and for some classes of square matrices, see \cite{Rump03SJMAA-Structured,Markovsky08A-Structured}) has no closed-form solution. 

Different optimization methods have been developed for different structures and approximation criteria, see~\cite{Markovsky08A-Structured} for a historical overview. 
Many optimization methods use the fact that the rank constraint $\rrank \struct (\sest{p}) \le r$ is equivalent to existence of a full row rank matrix $R \in \bbR^{d\times m}$ satisfying $R\struct (\sest{p}) = 0$, where $d \defeq m-r$ is the \textit{rank reduction} in \eqref{eq:prob_slra}. Problem \eqref{eq:prob_slra} then can be rewritten (for weighted $2$-norm) as  
\begin{equation}
\minim_{R \in \bbR^{d \times m}, \sest{p} \in \bbR^{\np}} \|\sdat{p} - \sest{p}\|^2_{\wmat} \quad \sto \; \rrank R = d \quad\text{and}\quad R \struct (\sest{p}) = 0.
\label{eq:ker_slra1}\tag{SLRA$_R$}
\end{equation}
Methods for \eqref{eq:ker_slra1} include Riemannian SVD \cite{DeMoor93LAA-Structured}, structured total least norm approach \cite{Rosen.etal96SJMAA-Total,Mastronardi.etal00SJMAA-Fast}, and methods based on the variable projection principle. Note that most of the optimization methods mentioned above were developed for
the \textit{structured total least squares} problem, which is the problem \eqref{eq:ker_slra1} with an additional constraint 
\begin{equation}
R = \bmx X & -\eyemat{d} \emx, \quad X \in \bbR^{d\times r}.
\label{eq:ker_slra11}\tag{STLS}
\end{equation}
Most of the methods listed above were designed only for special cases of the weighted $2$-norm \eqref{eq:wnorm}.

The \textit{variable projection} approach was proposed in \cite{GolubPereyra73SJoNA-Differentiation} for separable nonlinear least squares problems. Variable projection was first applied for some special cases of \eqref{eq:ker_slra1} \cite{Manton.etal03ITSP-geometry,Markovsky.etal04NLAA-computation}. 
In the variable projection approach, \eqref{eq:ker_slra1} is rewritten as
\begin{align}
&\minim_{R\in \bbR^{d\times m}, \rrank R = d } \costfun(R),\quad \text{where}& 
\label{eq:ker_slra}\tag{OUTER} \\
&\costfun(R) \defeq  
\Big( \min_{\sest{p} \in \bbR^{\np}} \|\sdat{p} - \sest{p}\|^2_{\wmat} \quad \sto \; R \struct(\sest{p}) = 0\Big).& \label{eq:ker_slra_red}\tag{$f(R)$}
\end{align}
The inner minimization problem \eqref{eq:ker_slra_red} is a \textit{least-norm problem} \cite{BoydVandenberghe04-Convex} and has a closed form solution. Therefore \eqref{eq:prob_slra} is reduced to optimization of \eqref{eq:ker_slra_red} on a space of dimension $dm$, which is typically much smaller than the dimension $\np$ of the eliminated variable $\sest{p}$.

The cost function $f(R)$ depends only on the subspace spanned by the rows of the argument~$R$, \ie~$f(R_1) = f(R_2)$ if $\rowspan R_1 = \rowspan R_2$.
Therefore, $f(R)$ can be considered as a function defined on the Grassmann manifold \cite{Absil.etal08-Optimization} of all $d$-dimensional subspaces of $\bbR^{n}$, and the problem \eqref{eq:ker_slra} is optimization on a Grassmann manifold. The optimization problem on a Grassmann manifold can be either transformed to an optimization problem on an Euclidean space (see \cite{Usevich.Markovsky12conf-Structured}), or solved by iterations in tangent spaces  (see, for example, \cite{Manton.etal03ITSP-geometry,Absil.etal08-Optimization}). Therefore, standard optimization routines can be used to minimize \eqref{eq:ker_slra}. 

The computation of the cost function has complexity $O(n^3)$ if the inner minimization problem \eqref{eq:ker_slra_red} is solved by general-purpose methods (\eg~the QR decomposition \cite{Golub.VanLoan96-Matrix}). For analytic computation of derivatives of $f(R)$, which can speed up the convergence of local optimization methods, only algorithms with complexity $O(n^3)$ were proposed in the case of general affine structure \cite{Borsdorf12-Structured}. 

In \cite{Markovsky.etal04NLAA-computation,Markovsky.etal06-Exact} it was shown that for a class of structures the cost function \eqref{eq:ker_slra_red} and its gradient can be evaluated in $O(mn)$ flops, which leads to efficient  local optimization methods for \eqref{eq:prob_slra}. The structures considered in \cite{Markovsky.etal04NLAA-computation,Markovsky.etal06-Exact} are of the form  $\rowvec{\struct^{(1)}}{\struct^{(q)}}^{\top}$, where each block $\struct^{(l)}$ is block-Toeplitz, block-Hankel or unstructured, and only whole blocks can be fixed. Only the $2$-norm approximation criterion was considered and the constraint  \eqref{eq:ker_slra11} on $R$ was required.



\subsection{Main results and composition of the paper}
In this paper we show that the cost function $f(R)$, its gradient and an approximation of its Hessian can be evaluated in $O(m^2 n)$ operations  for structure \eqref{eq:mosaic_like} and element-wise weighted $2$-norm (that allows fixed values constraints). If the weights are block-wise (or only whole blocks are fixed), the cost function and the gradient can be evaluated in $O(mn)$ operations, as in~\cite{Markovsky.VanHuffel05JCAM-High}.

We develop algorithms for evaluation of $f(R)$ and its gradient for general affine structure and arbitrary weighted 2-norm. The structure of $f(R)$ is derived in a similar way to~\cite{Markovsky.etal04NLAA-computation}, but in contrast to~\cite{Markovsky.etal04NLAA-computation,Markovsky.etal06-Exact} we do not use a probabilistic interpretation of~\eqref{eq:prob_slra}. Instead, we show how the matrix structure \eqref{eq:struct_def} is mapped to the structure of the cost function. In addition, our definition of the weighted 2-norm incorporates fixed values constraints, which also simplifies the derivation of the cost function structure in this case.

We provide an explicit derivation of Gauss-Newton approximations of the Hessian of $f(R)$, for the general affine structure and weighted $2$-norm. This derivation was omitted in~\cite{Markovsky.etal06-Exact} and other papers, but was used in the computational routines in~\cite{Markovsky.VanHuffel05JCAM-High}. We provide two variants of the approximation of the Hessian, based on two different representations of $f(R)$ as a sum of squares.

The paper is organized as follows. In Section~\ref{sec:gen_prop}, we 
review some basic properties of affine structures and \eqref{eq:prob_slra} in the weighted 2-norm.
 Section~\ref{sec:varpro} covers the derivation of $f(R)$ for the general affine structure and weighted 2-norm. For block structures, the cost function is expressed via cost functions for the blocks.
Based on results of Section~\ref{sec:varpro}, we develop general-purpose algorithms for evaluation of the cost function and its derivatives in Section~\ref{sec:algorithms}. Finally, in Section~\ref{sec:mosaic}, we specialize the algorithms from Section~\ref{sec:algorithms} for mosaic Hankel structures and derive their computational complexity. In Section~\ref{sec:gamma}, we discuss the conditions on $\struct$ and $r$ under which the algorithms in the paper are applicable; we also discuss accuracy of the computations.

\subsection{Notation}
We use lowercase letters for vectors and uppercase letter for matrices (square matrices are denoted by upright letters). 
By convention, blank elements in matrices denote zeros. Some of the notation used in the paper is summarized in Table~\ref{notation}. 
\begin{table}[htb!]
\centering
\begin{tabular}{p{3cm} p{0.5cm} p{8.5cm}}
$a_i, A_{i,j}$ &---& elements of vectors and matrices. \\
$k:l$ &---& the integer vector $\bmx k & \cdots & l\emx$; we also frequently use convention $i \in k:l$ to denote that an integer $i$  satisfies $k \le i \le l$.\\
$a_{\calI}, A_{\calI,\calJ}$ &---& sub-vectors and sub-matrices indexed by integer vectors; a single colon stands for selecting all indices, \eg~$A_{:,2}$ is the second column of the matrix $A$. \\
$0_{m\times n}$ &---& $m\times n$ matrix of zeros.\\ 
$1_{m\times n}$ &---& $m\times n$ matrix of ones.\\ 
$\eyemat{m}$ &---& $m \times m$ identity matrix $\bsm 1 & & \\ & \ddots& \\ & & 1 \esm$. \\
$\shiftmat{m}$ &---& $m \times m$ shift matrix,
\ie\;$\shiftmat{m} :=
\bsm
0_{(m-1)\times 1} & \eyemat{m-1} \\
0         & 0_{1\times (m-1)} \\
\esm.$ \\
$\ortvec{i}$ &---& $i$-th unit vector.\\ 
$\vcol(v^{(1)}, \ldots, v^{(N)})$ &---& concatenation of vectors $v^{(1)}, \ldots, v^{(N)}$.\\ 
$\blkdiag(A_1, \ldots, A_N)$ &---& block-diagonal matrix composed of matrices $A_1, \ldots, A_N$.\\ 
$\text{vec} A$   &---& the column-wise vectorization of a matrix.\\
$\wmat^{-1}$   &---& either the inverse of $\wmat$ if $\wmat$ is nonsingular, 
 or $\diag(w_1^{-1}, \ldots, w_n^{-1})$ if $\wmat = \diag(w_1,\ldots, w_n)$ and $w_k~\in~(0;\infty]$ (with the convention $\infty^{-1} = 0$). \\
$\wmat^{-\top}$   &---& the transpose of the inverse (\ie~$(\wmat^{-1})^{\top}$). \\
$\rchol{\mathrm{W}}$ &---& either the right Cholesky factor of $\mathrm{W}$ if $\mathrm{W}$ is positive definite, or  $\diag(\sqrt{w_1}, \ldots \sqrt{w_n})$ 
  if $\wmat = \diag(w_1,\ldots, w_n)$ and $w_k~\in~(0;\infty]$ (with the convention $\sqrt{\infty} = \infty$).\\
\end{tabular}
\caption{Notation}\label{notation}
\end{table}

\section{Affine structures and weighted $2$-norms}\label{sec:gen_prop}

\subsection{Basic properties}
Affine structures can  be defined in an equivalent to 
\eqref{eq:struct_def} \textit{vector form} 
\begin{equation}
\mvec \struct(p) = \mvec S_0 + \maffine{\struct} \, p, 
\label{eq:struct_vec}\tag{$\mvec \struct$}
\end{equation}
where $\maffine{\struct}$ is the matrix representation of the linear part of  \eqref{eq:struct_def} in the basis $\{\ortvec{i} (\ortvec{j})^{\top}\}^{m,n}_{i,j=1}$:
\begin{equation}
\maffine{\struct} \defeq \bmx \mvec S_1 & \cdots & \mvec S_{\np} \emx.
\label{eq:lin_matr}\tag{$\maffine{\struct}$}
\end{equation}

\begin{example}\label{ex:hankel23} $2 \times 3$ Hankel matrices $\HH_{2,3}(p)$ can be represented as \eqref{eq:struct_def} with
\[
S_0 = 0_{2\times 3}, \;
S_1 = \bmx 1 & 0 & 0 \\ 0 & 0 & 0\emx, \;
S_2 = \bmx 0 & 1 & 0 \\ 1 & 0 & 0\emx, \;
S_3 = \bmx 0 & 0 & 1 \\ 0 & 1 & 0\emx, \;
S_4 = \bmx 0 & 0 & 0 \\ 0 & 0 & 1\emx.
\]
In this case,
\[
\maffine{\struct} = 
\bmx 
1 &   &   &   \\ 
  & 1 &   &   \\ 
  & 1 &   &   \\ 
  &   & 1 &   \\ 
  &   & 1 &   \\ 
  &   &   & 1
\emx.
\] 

\end{example}
Using \eqref{eq:struct_vec}, it is easy to show that the Frobenius norm is a weighted 2-norm.
\begin{note}\label{not:frob}
Let \eqref{eq:struct_def} be an injective map (which corresponds to linearly independent $\{S_k\}_{k=1}^{\np}$ or \eqref{eq:lin_matr}  with full column rank). Then 
\[
\|\struct(\sest{p}) - \struct(\sdat{p})\|^2_{\rm F} = 
\|\maffine{\struct} (\sest{p} - \sdat{p})\|^2_2 = 
\| \sest{p} - \sdat{p}\|^2_{\wmat_1},
\]
where $\wmat_1 = \maffine{\struct}^\top \maffine{\struct}$ is the Gramian of the system of matrices $\{S_k\}_{k=1}^{\np}$.
\end{note}

\begin{example}\label{ex:hankel_w}
In Example~\ref{ex:hankel23}, the Gramian $\maffine{\struct}^\top \maffine{\struct}$ is diagonal, and the weighted norm corresponding to the Frobenius norm is given by \eqref{eq:wtoW} with $w = \vcol(1, 2, 2, 1)$.
\end{example}
It is not difficult to show that the Frobenius norm in Note~\ref{not:frob} can be replaced by any weighted Frobenius norm (a weighted 2-norm of $\mvec \left(\struct(\sest{p}) - \struct(\sdat{p})\right)$).

\subsection{From weighted $2$-norm to $2$-norm}
Any \eqref{eq:prob_slra} problem with weighted $2$-norm can be reduced to an unweighted \eqref{eq:prob_slra} problem. Consider the following change of parameters
\begin{equation}
\sest{p} = \sdat{p} - \rchol{\wmat}^{-1} \sdelt{p}.
\label{eq:repar_w}\tag{$\sdelt{p} \rightarrow \sest{p}$}
\end{equation}
We define the transformed structure in the vector form \eqref{eq:struct_vec}:
\begin{equation}
\mvec \structdelt(\sdelt{p}) \defeq  
\mvec \struct(\sdat{p}) - \maffine{\struct} \rchol{\wmat}^{-1} \sdelt{p},
\label{eq:struct_weight}\tag{$\structdelt$}
\end{equation}
By \eqref{eq:struct_vec}, \eqref{eq:struct_weight} satisfies 
\begin{equation}
\structdelt(\sdelt{p}) = \struct(\sest{p})
\label{eq:SdeqS}\tag{$\structdelt \leftrightarrow \struct$}
\end{equation}
for any $\Delta p$. If $\wmat$ in \eqref{eq:wnorm} is positive definite, then by $\wmat = \rchol{\wmat}^{\top}\rchol{\wmat}$ we have that
\begin{equation}
\|\sest{p} - \sdat{p}\|_{\wmat}^{2} =   \sdelt{p}^{\top} \rchol{\wmat}^{-\top} \wmat \rchol{\wmat}^{-1}  \sdelt{p} =  \|\sdelt{p} \|_2^2.
\label{eq:wnorm_2norm}\tag{$\|\cdot\|_{\wmat} = \|\cdot\|_2$}
\end{equation}
Note that \eqref{eq:wnorm_2norm} does not hold for \eqref{eq:wtoW}.
Nevertheless, the following result holds.

\begin{proposition}\label{prop:weighted_cases}
Any problem \eqref{eq:prob_slra} with \eqref{eq:wnorm} is equivalent to
\begin{equation}
\minim_{\sdelt{p} \in \bbR^{\np}} \| \sdelt{p} \|_2^2 \;  \sto \;
\rrank \structdelt(\sdelt{p}) \le r.
\label{eq:wslra_reduced}\tag{SLRA$_\Delta$}
\end{equation}
\end{proposition}
\proofapp

\section{Variable projection}\label{sec:varpro}

\subsection{Variable projection for the weighted $2$-norm}
In this subsection, we derive an explicit expression for the cost function \eqref{eq:ker_slra_red} for the general affine structure \eqref{eq:struct_def} and weighted $2$-norm \eqref{eq:wnorm}. Consider the change of variables \eqref{eq:repar_w}. By Proposition~\ref{prop:weighted_cases} and \eqref{eq:SdeqS}, the optimization problem in \eqref{eq:ker_slra_red} is equivalent to
\begin{equation}
\begin{split}
&\minim_{\sdelt{p} \in \bbR^{\np}} \|\sdelt{p}\|_2^2 \\
&\sto R \structdelt(\sdelt{p}) = 0, 
\end{split}\label{eq:lin_least_norm0}\tag{LN$_\Delta$}
\end{equation}
By \eqref{eq:struct_weight} and properties of the vectorization operator, \eqref{eq:lin_least_norm0}
can be rewritten as 
\begin{equation}
\begin{split}
&\minim_{\sdelt{p} \in \bbR^{\np}} \|\sdelt{p}\|_2^2 \\
&\sto G(R) \Delta p  = \resd{R},
\end{split}\label{eq:lin_least_norm}\tag{LN}
\end{equation}
where $\resd{R} \in \bbR^{dn}$ and $G(R) \in \bbR^{dn \times \np}$ are defined as
\begin{align}
&\resd{R} \defeq  \mvec \left( R \struct(\sdat{p}) \right) = (\eyemat{n} \otimes R) \mvec \struct(\sdat{p}), \label{eq:resd}\tag{$\resd{R}$}\\
&G(R) \defeq (\eyemat{n} \otimes R) \maffine{\struct} \rchol{\wmat}^{-1} = 
\bmx \mvec \left( R S_1\right) & \cdots & \mvec \left( R S_{\np}\right) \emx \rchol{\wmat}^{-1}.
\label{eq:lin_matr2}
\tag{$G(R)$}
\end{align}
Problem \eqref{eq:lin_least_norm} is a linear \textit{least-norm problem} in a standard form \cite[Ch.~6]{BoydVandenberghe04-Convex}, and has a closed-form solution. If $G(R)$ is of full row rank, the solution of \eqref{eq:lin_least_norm} exists and is given by
\begin{align}
&\sdelt{p}^{*} (R) =   G^{\top} (R) \Gamma^{-1}(R)  \resd{R},  
\label{eq:ker_slra_cos2t} \tag{$\sdelt{p}^{*}(R)$}\\
&\costfun(R) := \|\Delta p^{*}(R)\|^2_2 = {\resds^\top(R)} \Gamma^{-1}(R)  \resd{R},
\label{eq:f_unweighted}\tag{$\resds^\top \Gamma^{-1}\resds$}
\end{align}
where
\begin{equation}
\Gamma = \Gamma(R) :=  G(R) G^{\top} (R) \in \bbR^{dn\times dn}.
\label{eq:GammaG}\tag{$\Gamma(R)$}
\end{equation}

\begin{note}
If $G(R)$ is not of full row rank but the problem \eqref{eq:lin_least_norm} is feasible, the solution of \eqref{eq:lin_least_norm} is given by replacing the inverse of $\Gamma$ with its pseudoinverse in \eqref{eq:ker_slra_cos2t} and \eqref{eq:f_unweighted}.
\end{note}
In what follows, we assume that $G(R)$ is of full row rank, which is equivalent to $\Gamma(R)$  being invertible. A necessary condition for this is
\begin{equation}
\np \ge nd,
\label{eq:NC}\tag{NC}
\end{equation}
i.e. the number of rows of $G(R)$ should not exceed the number of its columns.  We discuss the invertibility  conditions of $\Gamma$ for mosaic Hankel matrices in Section~\ref{sec:gamma}.

\begin{note}
Since $\Gamma$ is a polynomial matrix (in the entries of $R$), it can be inverted symbolically  if it is nonsingular at least for one $R$. In this case, $\Gamma$ is infinitely differentiable for all $R$ such that $\Gamma(R)$ is nonsingular
\cite{Usevich.Markovsky12conf-Structured}.
\end{note}

\subsection{Block partitioning of $\Gamma(R)$}

Next, we show how the computation of $\Gamma$ can be simplified. Let us define 
\begin{equation}
\cmat = \cmat_{\struct, \wmat} \defeq \maffine{\struct} \wmat^{-1} \maffine{\struct}^\top \in \bbR^{mn\times mn}.
\label{eq:cmat_struct}\tag{$\cmat$}
\end{equation}
Consider the following block partitioning of $\Gamma$ and $\cmat$:
\begin{equation*}
\Gamma(R)  =
\bmx 
\mblk{\Gamma}{1}{1} (R) & \cdots  & \mblk{\Gamma}{1}{n} (R) \\
\vdots           &        & \vdots                 \\
\mblk{\Gamma}{n}{1} (R) & \cdots  & \mblk{\Gamma}{n}{n} (R) \\
\emx, \quad
\cmat = 
\bmx 
\mblk{\cmat}{1}{1}  & \cdots  & \mblk{\cmat}{1}{n} \\
\vdots       &        & \vdots      \\
\mblk{\cmat}{n}{1}  & \cdots  & \mblk{\cmat}{n}{n} \\
\emx,
\end{equation*}
where   $\mblk{\Gamma}{i}{j}(R) \in \bbR^{d \times d}$ and $\mblk{\cmat}{i}{j} \in \bbR^{m \times m}$. Then, 
$\mblk{\Gamma}{i}{j}(R)$ and $\mblk{\cmat}{i}{j}$ are related by
\begin{equation}
 \mblk{\Gamma}{i}{j} (R) \defeq R \mblk{\cmat}{i}{j} R^{\top}.
\label{eq:gamma_block}\tag{$\mblk{\Gamma}{i}{j}$}
\end{equation}
Indeed, from \eqref{eq:GammaG} and \eqref{eq:lin_matr2} and the fact that $\rchol{\wmat}^{-1}\rchol{\wmat}^{-\top} = {\wmat^{-1}}$ we have that
\[
\Gamma(R) = (\eyemat{n} \otimes R) \cmat (\eyemat{n} \otimes R^{\top}) = 
\blkdiag(R, \ldots, R) \cmat \blkdiag(R^{\top}, \ldots, R^{\top}),
\]
which proves \eqref{eq:gamma_block}.

\begin{note}\label{note:pinvgamma}
The structure of $\Gamma$ is determined by the structure of $\cmat$. For example, if $\cmat$ is block-sparse, block-banded or block-Toeplitz, then $\Gamma$ has the same structural property. The structural properties of $\cmat$ depend on the matrix structure $\struct$ and the weight matrix $\wmat$.
\end{note}

A partition of type \eqref{eq:gamma_block}  was derived in \cite{Markovsky.etal04NLAA-computation,Markovsky.etal06-Exact} from statistical considerations for a class of structures. Here we have derived it for general $\struct$ and $\wmat$ via a series of algebraic transformations. 

\begin{example}\label{ex:hankel4}
In Example~\ref{ex:hankel_w}, $G(R)$ has the form
\[
G(R) =
\bmx 
R_{1,1} & R_{1,2} &         &   \\ 
        & R_{1,1} & R_{1,2} &   \\ 
        &         & R_{1,1} & R_{1,2} \\ 
\emx
\diag(1,\sqrt{2}, \sqrt{2}, 1)
\]
for $R \in \bbR^{1\times 2}$. In addition, $\cmat$ has the form
\[
\cmat_{\struct,\wmat} 
= 
\bmx 
1 &   &   &   &   &   \\ 
  & 2 & 2 &   &   &   \\ 
  & 2 & 2 &   &   &   \\ 
  &   &   & 2 & 2 &   \\ 
  &   &   & 2 & 2 &   \\ 
  &   &   &   &   & 1
\emx
\] 
\end{example}
Examples~\ref{ex:hankel23}~and~\ref{ex:hankel4} will be generalized in Section~\ref{sec:scal_hank} to $\HH_{m,n}$ with arbitrary $m$ and $n$.

\subsection{Variable projection for block matrices}\label{sec:varpro_tools}
In this subsection, we consider basic transformations of structures and derive the form of \eqref{eq:ker_slra_red} for these transformed structures. First, we consider striped and layered block matrices.
\begin{lemma}[Striped structure]\label{lem:striped}
Let $p = \vcol (p^{(1)}, \ldots, p^{(N)})$, with $p^{(l)} \in \bbR^{\np^{(l)}}$ and 
\begin{equation}
\struct(p) = \bmx\struct^{(1)}(p^{(1)}) & \cdots & \struct^{(N)}(p^{(N)}) \emx,
\label{eq:S_striped}\tag{striped $\struct$}
\end{equation}
be the striped structure,  $\struct^{(l)}: \bbR^{\np^{(l)}} \to \bbR^{m \times n_l}$.
Then for $\struct$ and $\wmat = \blkdiag(\wmat^{(1)},\ldots,\wmat^{(N)})$, $\wmat^{(l)} \in \bbR^{\np^{(l)}\times\np^{(l)}}$, we have that:
\begin{enumerate}
\item the cost function \eqref{eq:ker_slra_red} is equal to the sum 
\begin{equation}
\costfun(R)  
= \sum\limits_{l=1}^N \costfun^{(l)}(R),
\label{eq:f_striped}\tag{striped $\costfun$}
\end{equation} 
where $\costfun^{(l)}$ is the cost function \eqref{eq:ker_slra_red} for the structure $\struct^{(l)}$ and the weight matrix $\wmat^{(l)}$;

\item the optimal correction $\Delta p^{*}(R)$ is the concatenation 
\begin{equation}
\sdelt{p}^{*}(R) = 
\vcol\big(\Delta p^{(1)*} (R), \ldots, \Delta p^{(N)*} (R)\big)
\label{eq:corr_striped}\tag{striped $\sdelt{p}^{*}$}
\end{equation}
of the corrections \eqref{eq:ker_slra_cos2t} for the structures $\struct^{(l)}$ and weight matrices $\wmat^{(l)}$.
\end{enumerate}
\end{lemma}
\begin{proof}
The inner minimization problem \eqref{eq:lin_least_norm0} can be expressed as
\[
\minim_{\Delta p^{(l)} \in \bbR^{\np^{(l)}}} \sum\limits_{l=1}^{N} \|\Delta p^{(l)}\|_2^2
\;\sto\;  R \structdelt^{(l)} (\Delta p^{(l)}) = 0, \quad\text{for }\quad l \in 1:N.
\]
This sum of squares is minimized by \eqref{eq:corr_striped}, and its norm is given by \eqref{eq:f_striped}.
\end{proof}

\begin{lemma}[Layered structure]\label{lem:layered}
Let $p = \vcol (p^{(1)}, \ldots, p^{(q)})$, with $p^{(k)} \in \bbR^{\np^{(k)}}$ and 
\begin{equation*}
\struct(p) = \bmx\struct^{(1)}(p^{(1)}) \\ \vdots \\ \struct^{(q)}(p^{(q)}) \emx,
\end{equation*}
be a \textit{layered structure}, $\struct^{(k)}: \bbR^{\np^{(k)}} \to \bbR^{m_k \times n}$.
Then for $\struct$  and 
$\wmat = \blkdiag(\wmat^{(1)},\ldots,\wmat^{(q)})$, $\wmat^{(k)} \in \bbR^{\np^{(k)}\times\np^{(k)}}$
we have that:
\begin{enumerate}
\item\label{it:g_layered} the $G$ matrix is 
\[
G(R) = \bmx G^{(1)}(R^{(1)}) & \cdots &  G^{(q)}(R^{(q)})\emx,
\]
where 
\begin{equation}
R =\bmx R^{(1)} & \cdots & R^{(q)} \emx
\label{eq:Rpart}\tag{$R^{(k)}$}
\end{equation}
is the partition of $R$ into $R^{(k)} \in \bbR^{d \times m_k}$;

\item the $\Gamma$ matrix is equal to the sum 
\[
\Gamma(R) = \sum\limits_{k=1}^q \Gamma^{(k)}(R^{(k)})
\]
of the matrices \eqref{eq:GammaG} corresponding to \eqref{eq:prob_slra} with $\struct^{(k)}$ and $\wmat^{(k)}$;

\item the matrix \eqref{eq:cmat_struct}  for $\struct$ and $\wmat$ is composed of the blocks
\[
\mblk{\cmat}{i}{j} = \blkdiag(\mblk{\cmat^{(1)}}{i}{j},\ldots,\mblk{\cmat^{(q)}}{i}{j}),
\]
where $\cmat^{(k)}$ is the matrix \eqref{eq:cmat_struct} for the structure $\struct^{(k)}$ and the weight matrix $\wmat^{(k)}$;

\item the optimal correction $\Delta p^{*}(R)$ can be expressed as 
\[
\Delta p^{*}(R) = \vcol\big(\Delta p^{(1)*} (R), \ldots, \Delta p^{(N)*} (R)\big),
\]
where $\sdelt{p}^{(k)*} (R) \defeq (G^{(k)}(R^{(k)}))^{\top} \Gamma^{-1}(R) \resd{R}$.
\end{enumerate}
\end{lemma}
\begin{proof}
It is sufficient to prove statement~\ref{it:g_layered}, which follows from
\[
\begin{split}
\mvec(R \structdelt (\sdelt{p})) &= 
\sum\limits_{k=1}^q \mvec\big(R^{(k)} \structdelt^{(k)} (\sdelt{p}^{(k)})\big)  \\
&=\sum\limits_{k=1}^q \left(\mvec \big(R^{(k)} \struct^{(k)}(\sdat{p}^{(k)})\big) - G^{(k)} (R^{(k)}) \sdelt{p}^{(k)}\right) \\
&= 
\mvec \big(R \struct(\sdat{p})\big) - \bmx G^{(1)}(R^{(1)}) & \cdots &  G^{(q)}(R^{(q)})\emx \sdelt{p}.
\end{split}
\]
The other statements follow from \eqref{eq:GammaG} and \eqref{eq:cmat_struct}.
\end{proof}
Next we examine the effect of left multiplication of the structure by a matrix of full row rank.
\begin{lemma}[Multiplication by $\Phi$]\label{lem:mulphi}
Let $\struct$ be defined as
\[
\struct(p) = \Phi \struct'(p),
\] 
where $\struct': \bbR^{\np} \to \bbR^{m' \times n}$ and $\Phi \in \bbR^{m \times m'}$ is such that $\rrank \Phi = m \le m'$. Then, 
\[
\costfun_\struct(R) = \costfun_{\struct'}(R \Phi).
\]
\end{lemma} 
\begin{proof}
This property is easily verified by rewriting the constraint in \eqref{eq:ker_slra_red} as
\[
R \struct(\sest{p}) = R \Phi\struct'(\sest{p}) = 0,
\]
where $\rrank R \Phi = m$.
\end{proof}

\section{Cost function and derivatives evaluation}\label{sec:algorithms}

In this section, we develop algorithms for computing the cost function \eqref{eq:ker_slra_red}, optimal correction \eqref{eq:ker_slra_cos2t} and their derivatives, for the general affine structure \ref{eq:struct_def}. The algorithms can be specialized for a specific class of structures by deriving the form of $\mblk{\cmat}{i}{j}$, as it will be done in Section~\ref{sec:mosaic} for the mosaic Hankel structure.

In Section~\ref{sec:cost_fun_corr}, we provide algorithms for evaluation of \eqref{eq:ker_slra_red}, computation of the optimal correction \eqref{eq:ker_slra_cos2t}, and computation of the gradient of \eqref{eq:ker_slra_red}.
In Sections~\ref{sec:nls_corr} and \ref{sec:nls_cholesky}, we consider the following approximation of the Hessian of \eqref{eq:ker_slra_red}. Let $f(R) = \|g(R)\|_2^2$, where $g(R) \in \bbR^{n_{\mathtt{s}}}$ with $n_{\mathtt{s}} \ge dm$, and $J_{g}$ be the Jacobian of $g$. Then, $2 J_{g}^{\top} J_{g}$ is an approximation of the Hessian of $f(R)$. Substituting this approximation in the Newton method is equivalent to performing Gauss-Newton iterations, and we will call it a \textit{Gauss-Newton approximation} of the Hessian. The Gauss-Newton approximation is frequently used in trust-region methods for solving nonlinear least-squares problems, \eg~in the Levenberg-Marquardt algorithm \cite{Marquardt63SJAM-algorithm}. In Section~\ref{sec:nls_corr}, we consider the Jacobian of $g(\cdot) = \sdelt{p}^{*} (\cdot)$, where $\sdelt{p}^{*} (\cdot)$ is defined in \eqref{eq:ker_slra_cos2t}. In Section~\ref{sec:nls_cholesky} we consider $g(\cdot) = \rchol{\Gamma}^{-\top} \resds (\cdot)$, where $\rchol{\Gamma}$ is the right Cholesky factor of $\Gamma$.

We will frequently use the following notation for the solution of the system $\Gamma \invgresd  = \resd{R}$:
\begin{equation}
\invgresd =  \invgresd(R) \defeq \Gamma^{-1} \resd{R},
\label{eq:invgresd}\tag{$\invgresd$}
\end{equation}


\subsection{Cost function, correction and gradient}\label{sec:cost_fun_corr}
From \eqref{eq:invgresd} and \eqref{eq:f_unweighted}, the cost function \eqref{eq:ker_slra_red} can be represented as
\begin{equation}
\costfun(R) = \invgresd^{\top}(R) \resd{R} = \resd{R}^{\top} \invgresd(R),
\label{eq:invgresd_resd}\tag{$\invgresd^{\top} \resds$}
\end{equation}
and  can be computed by Algorithm~\ref{alg:cost_fun}.
\begin{algorithm}[Cost function evaluation] \label{alg:cost_fun}
\quad\quad\quad\quad

\begin{algorithmic}[1]
\REQUIRE $\struct$, $\wmat$, $\mblk{\cmat}{i}{j}$, $\sdat{p}$, $R$
\STATE Compute \eqref{eq:resd}. 
\STATE Compute \eqref{eq:gamma_block} using $\mblk{\cmat}{i}{j}$.
\STATE Compute \eqref{eq:invgresd}.
\STATE Compute the cost function $f(R)$ as \eqref{eq:invgresd_resd}.
\ENSURE $f(R)$
\end{algorithmic}
\end{algorithm}

The term $\sdelt{p}^{*} (R)$ in \eqref{eq:ker_slra_cos2t} can be evaluated by Algorithm~\ref{alg:correction}.

\begin{algorithm}[Correction computation] \label{alg:correction}
\quad\quad\quad\quad

\begin{algorithmic}[1]
\REQUIRE $\struct$, $\wmat$, $\mblk{\cmat}{i}{j}$, $\sdat{p}$, $R$
\STATE Perform steps 1--3 of Algorithm~\ref{alg:cost_fun}.
\STATE Set $\sdelt{p}^{*}(R) = G^{\top}(R) \invgresd$.
\ENSURE $\sdelt{p}^{*}(R)$
\end{algorithmic}
\end{algorithm}

\begin{note}
The optimal $\sest{p}$ in \eqref{eq:ker_slra_red} can be computed by combining Algorithm~\ref{alg:correction} and \eqref{eq:repar_w}.
\end{note}



Instead of the gradient $\nabla \costfun$, for convenience, we use the \textit{matrix gradient} $\mgrad{d}{m} \costfun \in \bbR^{d \times m}$, defined as
\[
\mvec ( \mgrad{d}{m} \costfun) \defeq \nabla \costfun.
\] 
\begin{proposition}\label{prop:gradient}
Let $\invgresdmat\in \bbR^{d \times n}$  be the matrix constructed from the \eqref{eq:invgresd} as follows:
\begin{equation}
\invgresd = \mvec \invgresdmat.
\label{eq:invgresdmat}\tag{$\invgresdmat$}
\end{equation}
Then the matrix gradient is given by
\begin{equation}
\mgrad{d}{m}(\costfun) = 2  \invgresdmat \struct^\top(p) - 
2 \sum\limits_{i,j=1}^{m} \invgresdbl{j} \invgresdbl{i}^\top R \mblk{\cmat}{i}{j}.
\label{eq:mgrad_f}\tag{$\mgrad{d}{m}$}
\end{equation}
\end{proposition}
\proofapp

From Proposition~\ref{prop:gradient}, the gradient can be computed by Algorithm~\ref{alg:gradient}.
\begin{algorithm}[Gradient evaluation]\label{alg:gradient} 
\quad\quad\quad\quad

\begin{algorithmic}[1]
\REQUIRE $\struct$, $\wmat$, $\mblk{\cmat}{i}{j}$, $\sdat{p}$, $R$
\STATE Perform steps 1--3 of Algorithm~\ref{alg:cost_fun}.
\STATE Compute the first term $2\invgresdmat \struct^\top(\sdat{p})$ of the gradient \eqref{eq:mgrad_f}.
\STATE Compute the second term of the gradient \eqref{eq:mgrad_f}.
\ENSURE $\mgrad{d}{m} \costfun(R)$
\end{algorithmic}
\end{algorithm}

\subsection{Approximation of the Hessian: Jacobian of $\sdelt{p}^{*}$}\label{sec:nls_corr}
The following proposition gives an expression to compute the Jacobian of \eqref{eq:ker_slra_cos2t}.
\begin{proposition}\label{prop:jacobian}\quad\quad\quad\quad\quad\quad\quad\quad
\begin{itemize}
\item
The Jacobian of \eqref{eq:ker_slra_cos2t} is given by
\begin{equation}
\frac{\partial \sdelt{p}^{*}}{\partial R_{ij}} =
 G^{\top} (R) \Gamma^{-1} z_{ij} + \frac{\partial G^{\top}}{\partial R_{ij}}  \invgresd,
\label{eq:jacobian}\tag{${\partial \sdelt{p}^{*}}/{R_{ij}}$}
\end{equation}
where
\begin{equation}
z_{ij} \defeq{\frac{\partial \resds}{\partial R_{ij}} -
 \frac{\partial \Gamma}{\partial R_{ij}} \invgresd}.
\label{eq:zij}\tag{$z_{ij}$} 
\end{equation}

\item
The vector $z_{ij}$ can be expressed as 
\[
z_{ij} =  (\struct^{\top}(\sdat{p}))_{:,j}\otimes \ortvec{i} -
\left((z^{(1)}_j) \otimes \ortvec{i} + z^{(2)}_{ij}\right),
\]
where
\begin{align}
z^{(1)}_j &\defeq 
 \sum_{k=1}^{n} \vcol \left( (\ortvec{j})^{\top} \mblk{\cmat}{1}{k} R^{\top} \invgresdbl{k},
\ldots, 
(\ortvec{j})^{\top} \mblk{\cmat}{n}{k} R^{\top} \invgresdbl{k} \right),
\label{eq:zij1}\tag{$z^{(1)}_j$} \\
z^{(2)}_{ij} &\defeq \sum_{k=1}^{n} \invgresdmat_{i,k}
\vcol \left( R  \mblk{\cmat}{1}{k} \ortvec{j}, \ldots,  R  \mblk{\cmat}{n}{k} \ortvec{j}  \right).
\label{eq:zij2}\tag{$z^{(2)}_{ij}$}
\end{align}
\end{itemize}
\end{proposition}
\proofapp

Proposition~\ref{prop:jacobian} leads to the following algorithm.
\begin{algorithm}[Jacobian evaluation]\label{alg:jacobian}
\quad\quad\quad\quad

\begin{algorithmic}[1]
\REQUIRE $\struct$, $\wmat$, $\mblk{\cmat}{i}{j}$, $\sdat{p}$, $R$
\STATE Perform steps 1--3 of Algorithm~\ref{alg:cost_fun}.
\FOR{$j \in 1:m$}
\STATE Compute \eqref{eq:zij1}.
\FOR{$i \in 1:d$}
\STATE Compute \eqref{eq:zij2} and \eqref{eq:zij}.
\STATE Set $x \leftarrow G^{\top} \Gamma^{-1} z_{ij}$.
\STATE Set  $\frac{\partial \sdelt{p}^{*}}{\partial R_{ij}} \leftarrow x + \frac{\partial G^{\top}}{\partial R_{ij}}  \invgresd$.
\ENDFOR
\ENDFOR
\ENSURE Jacobian of $\sdelt{p}^{*}(R)$
\end{algorithmic}
\end{algorithm}

\subsection{Using Cholesky factorization}\label{sec:nls_cholesky}
In this section, we show that the cost function and an approximation of the Hessian can be computed using the Cholesky factorization of $\Gamma$
\[
\Gamma = \rchol{\Gamma}^\top \rchol{\Gamma}.
\]
The Cholesky factorization yields a numerically reliable way \cite[Ch. 4]{Golub.VanLoan96-Matrix} to solve the system of equations $\Gamma u = v$, by using the following identity
\[
\Gamma^{-1} = \rchol{\Gamma}^{-1} \rchol{\Gamma}^{-\top}.
\] 

\begin{algorithm}[Solve system $\Gamma u = v$]\label{alg:solve_cholesky} 
\quad\quad\quad\quad

\begin{algorithmic}[1]
\REQUIRE $\Gamma$, $v$
\STATE Compute the Cholesky factor $\rchol{\Gamma}$  of $\Gamma$.
\STATE Solve $\rchol{\Gamma}^{\top}  g= v$ and $\rchol{\Gamma}  u = g$ by backward substitution.
\ENSURE $u$
\end{algorithmic}
\end{algorithm}

Moreover, the cost function can be represented as
\begin{equation}
\costfun(R) = \|\rchol{\Gamma(R)}^{-\top} \resd{R}\|^{2}_2,
\label{eq:cost_lschol}\tag{$\|\rchol{\Gamma}^{-\top} \resds\|_2^2$}
\end{equation}
which leads to a more efficient algorithm for the cost function evaluation than Algorithm~\ref{alg:cost_fun}.

\begin{algorithm}[Cost function evaluation using Cholesky factorization]\label{alg:cost_fun_cholesky} 
\quad\quad\quad\quad

\begin{algorithmic}[1]
\REQUIRE $\struct$, $\wmat$, $\mblk{\cmat}{i}{j}$, $\sdat{p}$, $R$
\STATE Compute \eqref{eq:resd}.
\STATE Compute \eqref{eq:gamma_block} using $\mblk{\cmat}{i}{j}$.
\STATE Compute the Cholesky factor $\rchol{\Gamma}$.
\STATE Solve $\rchol{\Gamma}^{\top}  \lschol(R) = \resd{R}$.
\STATE Compute $f(R) = \|\lschol(R)\|_2^2$.
\ENSURE $f(R)$
\end{algorithmic}
\end{algorithm}
In addition, equation \eqref{eq:cost_lschol} defines a representation of the cost function as a sum of squares, which is more compact and easier to compute than \eqref{eq:f_unweighted}. Indeed, $\sdelt{p}^*(R) \in \bbR^{\np}$, but $\lschol(R) \in \bbR^{nd}$.
However, the elements of the Jacobian of $\lschol$
\[
\frac{\partial\lschol}{\partial R_{ij}}   =
\rchol{\Gamma(R)}^{-\top} \frac{\partial \resds}{\partial R_{ij}} +
\frac{\partial \rchol{\Gamma(R)}^{-\top}}{\partial R_{ij}} \resd{R},
\]
cannot be computed analytically due to the need of differentiation of $\rchol{\Gamma}^{-\top}$. For this purpose the \textit{pseudo-Jacobian}, proposed in~\cite{Guillaume.Pintelon96ITSP-Gauss}, can be used 
\begin{equation}
\frac{\pseudopartial\lschol}{\partial R_{ij}} \defeq
\left(\rchol{\Gamma(R)}^{-\top}  \frac{\partial \resd{R}}{\partial R_{ij}} -
\frac{1}{2} \rchol{\Gamma(R)}^{-\top} \frac{\partial \Gamma}{\partial R_{ij}} \Gamma^{-1} \resd{R} \right).
\label{eq:pseudojacobian}\tag{${\pseudopartial\lschol}/{\partial R_{ij}}$}
\end{equation}
\begin{note}
In~\cite{Guillaume.Pintelon96ITSP-Gauss} it was shown that for functions of the form \eqref{eq:f_unweighted}, the pseudo-Jacobian $J^{(\mathfrak{p})}_g$ yields the same stationary points of $\costfun(R)$ (i.e., $\nabla f = 2 (J^{(\mathfrak{p})}_g)^{\top} g$) and provides an approximation of the Hessian of $\costfun(R)$.
\end{note}

It is easy to see that \eqref{eq:pseudojacobian} and \eqref{eq:zij} differ only by a constant factor in one summand. Therefore,
\begin{equation}
\begin{split}
&\frac{\pseudopartial\lschol}{\partial R_{ij}} = \rchol{\Gamma(R)}^{-\top} \pseudoz_{ij}, \\
&\pseudoz_{ij} \defeq  (\struct^{\top}(\sdat{p}))_{:,j} \otimes \ortvec{i} -
\frac{1}{2}\left((z^{(1)}_j) \otimes \ortvec{i} + z^{(2)}_{ij}\right).
\end{split}
\label{eq:pseudozij}\tag{$\pseudoz_{ij}$}
\end{equation}
The resulting algorithm is Algorithm~\ref{alg:pseudojacobian}.

\begin{algorithm}[Pseudo-Jacobian evaluation]\label{alg:pseudojacobian}
\quad\quad\quad\quad

\begin{algorithmic}[1]
\REQUIRE $\struct$, $\wmat$, $\mblk{\cmat}{i}{j}$, $\sdat{p}$, $R$
\STATE Compute steps 1--2 of Algorithm~\ref{alg:gradient} using  Algorithm~\ref{alg:solve_cholesky}.
\FOR{$j \in 1:m$}
\STATE Compute \eqref{eq:zij1}.
\FOR{$i \in 1:d$}
\STATE Compute \eqref{eq:zij2} and \eqref{eq:pseudozij}.
\STATE Solve $\rchol{\Gamma(R)}^{\top} {\pseudopartial\lschol}/{\partial R_{ij}} = \pseudoz_{ij}$.
\ENDFOR
\ENDFOR
\ENSURE The pseudo-Jacobian $J^{(\mathfrak{p})}_{g_s}$
\end{algorithmic}
\end{algorithm}

\section{Weighted mosaic Hankel structured low-rank approximation problem}\label{sec:mosaic}
In this section, we establish the form of $\Gamma(R)$, $\mblk{\cmat}{i}{j}$ and $\costfun(R)$ for the structure \eqref{eq:mosaic_like} and weight matrix 
\begin{equation}
\wmat = \blkdiag(\wmat^{(1,1)},\ldots,\wmat^{(q,N)}).
\label{eq:wblockw}\tag{$\mathrm{blkdiag}\,\wmat$}
\end{equation}
 In view of Lemmae~\ref{lem:striped}--\ref{lem:mulphi}, we can consider only the scalar Hankel structure.


\subsection{Scalar Hankel matrices}\label{sec:scal_hank}

\begin{lemma}\label{lem:Shank}
For the scalar Hankel structure, the matrix \eqref{eq:lin_matr} 
 has the following form
\begin{equation}
\maffine{{\HH_{m,n}}} \defeq
\bmx
\bmx \eyemat{m}  & 0_{m\times (\np-m)} \emx \shiftmat{\np}^{0} \\
\bmx \eyemat{m}  & 0_{m\times (\np-m)} \emx \shiftmat{\np}^{1} \\
\vdots \\
\bmx \eyemat{m}  & 0_{m\times (\np-m)} \emx \shiftmat{\np}^{n-1}
\emx.
\label{eq:maffine_hank}\tag{$\maffine{{\HH_{m,n}}}$}
\end{equation}
\end{lemma}
\begin{proof}
Indeed, for a vector $p = \bmx p_1 & \ldots & p_{\np} \emx^{\top}$, we have
\begin{equation}
\bmx \eyemat{m}  & 0_{(\np-m)\times  m} \emx \shiftmat{\np}^{k} p = p_{\intrng{k+1}{k+m}}.
\label{eq:sel_subv}\tag{$[\eyemat{} \  0] \shiftmat{}^{k}$}
\end{equation}
Therefore, \eqref{eq:struct_vec} holds with \eqref{eq:maffine_hank} and $S_0 = 0_{m\times n}$. 
\end{proof}

\begin{corollary}
The matrix $G_{\HH_{m,n}}$ for $\wmat = \eyemat{m+n-1}$ and $R\in \bbR^{m \times d}$  has the form
\begin{equation}
G_{\HH_{m,n}} (R) =
\bmx
R_{:,1}    & R_{:,2}    & \ldots & R_{:,m}    &     &     & \\
       & R_{:,1}    & R_{:,2}    & \ldots & R_{:,m} &     & \\
       &        & \ddots &        &     &  \ddots  &    \\
       &        &        & R_{:,1}    & R_{:,2} & \ldots   & R_{:,m} 
\emx \in \bbR^{nd \times (m+n-1)}.
\label{G_h}\tag{$G_{\HH_{m,n}}$}
\end{equation}
\end{corollary}

We next consider the case of Hankel low-rank approximation with weighted $2$-norm.
\begin{proposition}
For the scalar Hankel structure and weight matrix $\wmat$, the blocks of \eqref{eq:cmat_struct} are equal to 
\[
\mblk{\cmat}{i}{j} =(\wmat^{-1})_{i:i+m-1,j:j+m-1}.
\]
\end{proposition}
\begin{proof}
From \eqref{eq:cmat_struct} and \eqref{eq:maffine_hank}, we have that 
\begin{equation}
\mblk{\cmat}{i}{j} = 
\bmx \eyemat{m}  & 0_{m\times (\np-m)} \emx \shiftmat{\np}^{i-1}
\wmat^{-1}
(\shiftmat{\np}^{\top})^{j-1}  
\bmx \eyemat{m}  \\ 0_{(\np-m)\times m} \emx.
\label{eq:v_w_hank_ij}\tag{$\mblk{\cmat}{i}{j}(\HH_{m,n})$}
\end{equation}
From \eqref{eq:sel_subv} one can verify that \eqref{eq:v_w_hank_ij} corresponds an $m \times m$ submatrix of~$\wmat^{-1}$.
\end{proof}

We will call a symmetric matrix \textit{$b$-banded ($b$-block-banded)} if all upper diagonals (resp. upper block diagonals) starting from $b$-th are zero. (By convention, the main (block) diagonal is the $0$-th upper (block) diagonal.) 

\begin{corollary}\label{cor:hank_cmat}
\begin{enumerate}
\item If $\wmat^{-1}$ is $b$-banded then $\cmat$ and $\Gamma$ are $(m+b-1)$-block-banded.

\item If $\wmat^{-1}$ is Toeplitz then $\cmat$ and $\Gamma$ are block-Toeplitz. 

\item For the case \eqref{eq:wtoW} $\cmat$ and $\Gamma$ are $m$-block-banded. In particular,

\begin{enumerate}
\item For $\wmat = \eyemat{\np}$, $\mblk{\cmat}{i}{j} = \cmat_{j-i}$, where
$\cmat_{k} \defeq (\shiftmat{m}^{\top})^{k}$ for $k \ge 0$, and $\cmat_{-k} = \cmat_k^{\top}$. (In this case, $\Gamma$ and $\cmat$ are  block-Toeplitz.)

\item For $\wmat = \diag(w_1, \ldots, w_{\np})$ (element-wise weights and fixed values),
\[
\mblk{\cmat}{i}{j} = 
\begin{cases} 
\diag(\gamma_{\intrng{i}{i+m-1}})
(\shiftmat{m}^{\top})^{j-i}, \quad& \text{for}\;j \ge i, \\
(\mblk{\cmat}{j}{i})^{\top}, \quad& \text{for}\;j < i,
\end{cases}
\]
where $\gamma_i \defeq w_i^{-1}$.
\end{enumerate}

\end{enumerate}
\end{corollary}

\subsection{The main theorem}
By Lemma~\ref{lem:striped}, we can consider only the layered Hankel structure
\begin{equation}
\HH_{\bfm, n } (p) = \bmx
\HH_{m_1,n}(p^{(1)})\\
\vdots \\
\HH_{m_q,n}(p^{(q)})
\emx,\quad \mbox{where } \bfm = \bmx m_1 & \cdots & m_q \emx.
\label{eq:layered}\tag{$\HH_{\bfm,n}$}
\end{equation}

\begin{theorem}\label{thm:main}
For the structure \eqref{eq:layered} and $\wmat = \blkdiag(\wmat^{(1)},\ldots,\wmat^{(q)})$, with $(\wmat^{(k)})^{-1}$ being $b^{(k)}$-banded, the following statements hold true:

\begin{enumerate}
\item the matrix $\Gamma$ is $\bw$-block banded with $d \times d$ blocks, with 
$\bw\defeq \max_{k} \{m_k+ {b^{(k)}} -1\}$;

\vspace{1em}

\item if $(\wmat^{(k)})^{-1}$ are all Toeplitz, then $\Gamma$ is block-Toeplitz;

\item for the case \eqref{eq:wtoW} and $\wmat^{(k)} = \diag w^{(k)}$, $w^{(k)} \in \bbR^{\np^{(k)}}$ the matrices $\cmat_{\HH_{\bfm,n}}$ and $\Gamma_{\HH_{\bfm,n}}(R)$ are block-banded with block bandwidth 
\[
\bw \defeq \max \{ m_l \}_{l = 1}^q; 
\]
in particular,
\begin{enumerate}
\item the matrices $\mblk{\cmat}{i}{j}$ for $j \ge i$ can be expressed as 
\begin{equation}
\mblk{\cmat}{i}{j} =   
\diag\left(\vcol\big(\gamma^{(1)}_{\intrng{i}{i+m_1-1}}, \ldots,
\gamma^{(q)}_{\intrng{i}{i+m_q-1}}\big)\right) (\shiftmat{\bfm}^{\top})^{j-i}, 
\label{Gamma_ij}\tag{$\mblk{\cmat}{i}{j}(\HH_{\bfm,n})$}
\end{equation}
where $\gamma^{(k)} = (w^{(k)})^{-1}$ and $\shiftmat{\bfm} \defeq \blkdiag(\shiftmat{m_1},\ldots,\shiftmat{m_q})$;

\item\label{it:toeplitz}
if the weights are constant for the blocks of \eqref{eq:layered}, \ie 
\begin{equation}
w^{(l)} \equiv \omega_l,\quad \mbox{where}\quad \omega_l \in \bbR,
\label{eq:blockwisew}\tag{block-wise $w$}
\end{equation}
then $\cmat$ is block-Toeplitz, \ie, $\mblk{\cmat}{i}{j} = \cmat_{j-i}$, and 
\[
\mblk{\Gamma}{i}{j} (R) = \Gamma_{j-i},\quad\Gamma_k = R \cmat_{k} R^\top,
\]
\[
\cmat_k = \blkdiag(\gamma_{1} \eyemat{m_1}, \ldots,\gamma_{q} \eyemat{m_q}) (\shiftmat{\bfm}^{\top})^{k} \;\mbox{and}\;
\quad \gamma_{l} \defeq \omega_{l}^{-1}.
\]

\end{enumerate}

\end{enumerate}
\end{theorem}
\begin{proof}
Theorem~\ref{thm:main} follows from Lemma~\ref{lem:layered} and Corollary \ref{cor:hank_cmat}.
\end{proof}

\subsection{Computational complexity of the algorithms for layered and mosaic Hankel structures}\label{sec:complexity}
In this section, we assume that $\wmat$ is given by \eqref{eq:wtoW}, and therefore $\Gamma$ matrix is $\bw$-block-banded with $\bw = \max_{k} \{m_k\}$. This case is implemented in the C++ solver of \cite{Markovsky.Usevich13JCAM-Software}.
In our complexity analysis, we count only the number of multiplications, since the number of additions is typically less than the number of multiplications. 

\begin{lemma}\label{lem:comp_banded_cholesky}
For \eqref{eq:layered}, the complexity of the steps in the Algorithm~\ref{alg:solve_cholesky} is given by $O(d^3 \bw^2 n)$ for step 1 and $O(d^2 \bw n)$ for step 2.
\end{lemma}
\begin{proof}
$\Gamma \in \bbR^{nd \times nd}$ is $\bw d$-banded, and the complexity of both steps is given in \cite[Ch. 4]{Golub.VanLoan96-Matrix}.
\end{proof}

\begin{theorem}\label{thm:cost_grad}
For \eqref{eq:layered}, the complexity of the cost function evaluation using Algorithm~\ref{alg:cost_fun_cholesky} and gradient evaluation using Algorithm~\ref{alg:gradient} is $O(d^3 \bw m n)$.
\end{theorem}
\proofapp

\begin{note}
By Lemma~\ref{lem:striped}, for the mosaic Hankel structure $\HH_{\bfm, \bfn}$ 
the complexity of the cost function and gradient evaluation is also $O(d^3 \bw m n)$.
\end{note}

\begin{note}\label{not:corr_comp}
The computation on the step 2 in the Algorithm~\ref{alg:correction}  has complexity $O(dmn)$.
\end{note}

\begin{theorem}\label{thm:jac_pj}
The computational complexity of Algorithm~\ref{alg:jacobian} (for the Jacobian) is $O(d^3 m^2 n)$. The computational complexity of Algorithm~\ref{alg:pseudojacobian} (for the pseudo-Jacobian) is $O(d^3 \bw m n)$.
\end{theorem}
\proofapp


Next, we show that the complexity is linear in $m$ (for fixed $d$ and $n$) for the cost function and the gradient in the case \eqref{eq:blockwisew}. In the cost function evaluation, the most expensive step is the Cholesky factorization, which can be performed in linear in $m$  number of operations in this case \cite{VanHuffel.etal04ICSM-High}. The computation of the gradient can be also simplified due to block-Toeplitz structure of $\Gamma$.
\begin{proposition}\label{prop:toeplitz_gradient}
Let $\invgresd$, $\invgresdmat$ be defined as in \eqref{eq:invgresd}, \eqref{eq:invgresdmat}, and weights satisfy
\eqref{eq:blockwisew}. Then the gradient \eqref{eq:mgrad_f} can be simplified to
\[
\mgrad{d}{m}(\costfun) = 2 \invgresdmat \struct^{\top}(\sdat{p}) - 
2\Big(N_0 R \cmat_0 + \sum\limits_{0 < k \le \min(n,\bw)} \big(N_k R \cmat_{k}  + N_k^\top R \cmat_{k}^{\top} \big)\Big),
\]
where $N_k \defeq \invgresdmat (\shiftmat{n}^{\top})^k (\invgresdmat)^{\top}$.
\end{proposition}
\proofapp

Using Proposition~\ref{prop:toeplitz_gradient}, the following theorem can be proved.

\begin{theorem}\label{thm:comp_blt}
For the case \eqref{eq:blockwisew}, the complexity of the cost function and gradient evaluation is equal to $O(d^3 m n)$.
\end{theorem}
\proofapp

\begin{note}
The computations can be further simplified for mosaic Hankel matrices \eqref{eq:mosaic} and block-wise weights, which are constant along the block rows of the mosaic Hankel matrix, \ie
\[
w = \vcol\left(
w^{(1,1)},\ldots,w^{(q,1)},
\ldots,
w^{(1,N)},\ldots,w^{(q, N)}
\right),\quad \mbox{where}\quad w^{(l,k)} \equiv \omega_l,\; \omega_l \in \bbR.
\]
Let the $\Gamma^{(l)}$ be the $\Gamma$ matrix for the structure $\HH_{\bfm, n_l}$, as in Lemma~\ref{lem:striped}. Then all $\Gamma^{(l)}$ are submatrices of $\Gamma^{(l_{max})}$, where $n_{l_{max}} \ge n_l$ for all $l$. Therefore, only the Cholesky factorization of $\Gamma^{(l_{max})}$ needs to be computed.
\end{note}

\begin{note}\label{not:wslra}
We have considered only the case when $\wmat$ is diagonal. However, by Theorem~\ref{thm:main}, $\Gamma$ is block-banded/block-Toeplitz if $\wmat^{-1}$ is $b$-banded. In this case the algorithms will have complexity similar to that in Theorems~\ref{thm:cost_grad}--\ref{thm:comp_blt}, with an additional factor $b$. 
\end{note}
Note~\ref{not:wslra} generalizes the results of \cite{Markovsky.VanHuffel06-Weighted}, where this type of weights was first considered.

\subsection{Numerical experiments}
First, we consider a family of $2\times 2$ mosaic Hankel matrices 
\begin{equation}
\HH_{\bsm m_1 & m_2 \esm, \bsm n_1 & n_2 \esm}, m_1 = 20, m_2 = 22, \bmx n_1 &n_2\emx = \bmx 250, 255\emx + 250 k.
\label{eq:ex1}\tag{EX1}
\end{equation}
In the case of $2$-norm ($\wmat =\eyemat{\np}$) we compute the cost function and its derivatives using the SLRA package \cite{Markovsky.Usevich13JCAM-Software}.
Hereafter, we use the term ``element-wise variant'' for the algorithms that treat weights as different values (implemented in the {\tt WLayeredHankelStructure} C++ class). We use the term ``block-wise variant'' for the algorithms that utilize Theorem~\ref{thm:comp_blt} (block-Toeplitz structure of $\Gamma$), and is implemented in the {\tt LayeredHankelStructure} C++ class. In these examples,we consider only the rank reduction $d =1$ ($r = m-1$).

\begin{figure}[!ht]
\begin{center}
\includegraphics[height=1.2in]{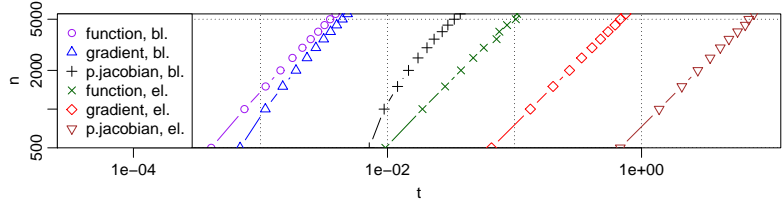}    
\caption{The number of columns $n$ versus computation time $t$, log-log scale. Computation times are shown for the cost function, gradient and pseudo-Jacobian, depending on $n$. ``el'' denotes element-wise variant and ``bl'' denotes block-wise variant.} 
\label{fig:ex1}                                 
\end{center}                                
\end{figure}

Fig.~\ref{fig:ex1} shows the time needed for computation of the cost function, gradient and pseudo-Jacobian for \eqref{eq:ex1}, with $k \in 0:10$. The computation time is plotted in logarithmic scale. In all cases the computation time is bounded by a linear function in $n$. This empirical observation is in agreement with Theorem~\ref{thm:cost_grad} and Theorem~\ref{thm:comp_blt}. 

Next, we show the computation time for varying $m$, for scalar Hankel matrices. In Fig.~\ref{fig:ex3} the computational time is plotted for element-wise and block-wise variants, as explained above. 


\begin{figure}[!ht]
\begin{center}
\subfigure{\includegraphics[height=1.4in]{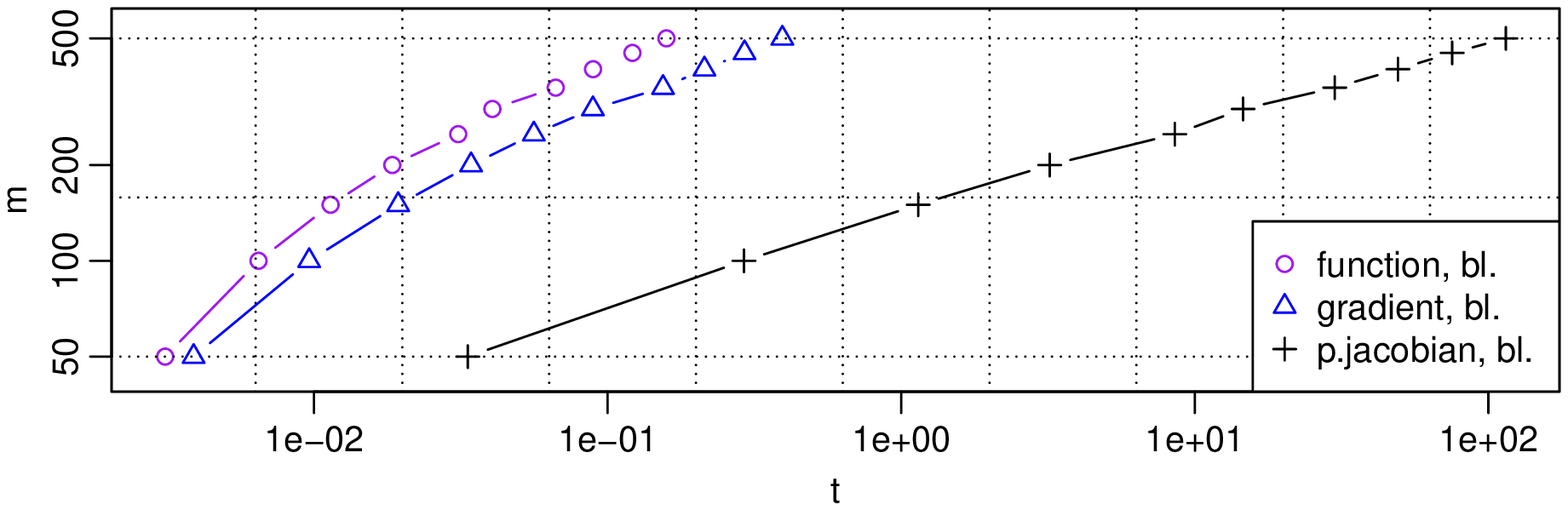}}\\ \subfigure{\includegraphics[height=1.5in]{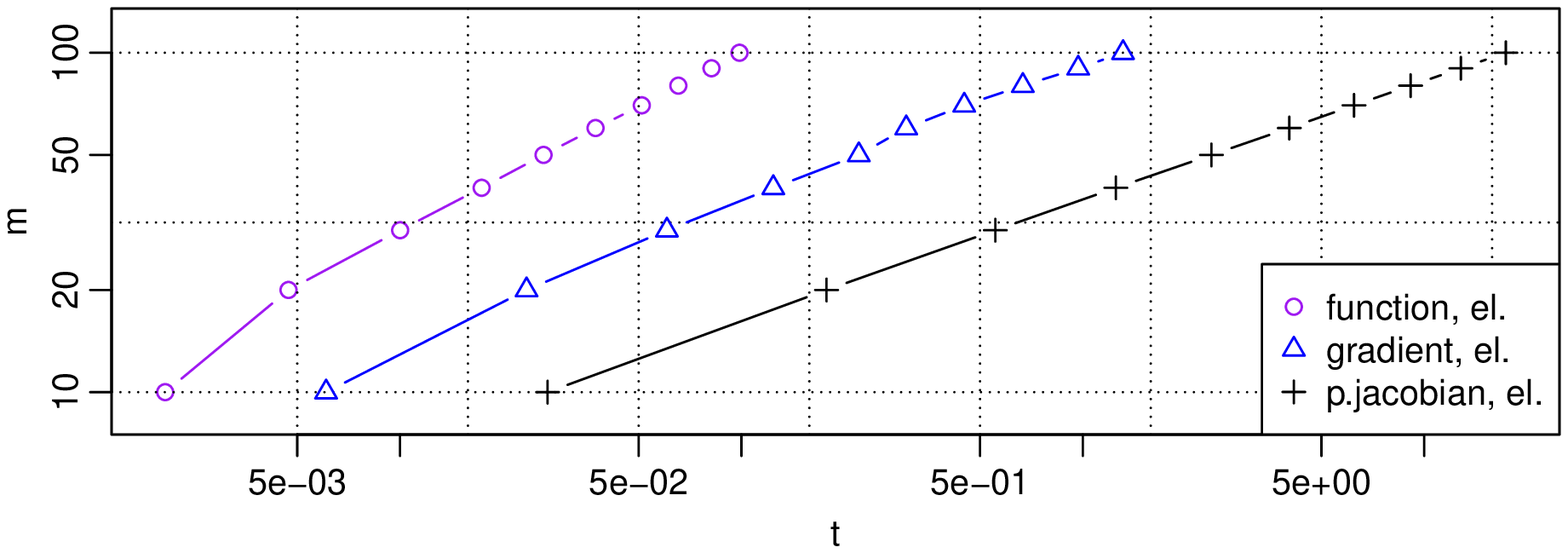}}
\caption{The number of rows $m$ versus computation time $t$, log-log scale. Computation times are shown  for the cost function, gradient and pseudo-Jacobian, for the structure $\HH_{m,2000}$. Top: block-wise variant, bottom: element-wise variant.} 
\label{fig:ex3}                                 
\end{center}                                
\end{figure}

Fig.~\ref{fig:ex3} shows that the computational time for the cost function and the gradient is growing faster than linearly in $m$ (in contrast to Theorems~\ref{thm:cost_grad} and~\ref{thm:comp_blt}). The computation of the pseudo-Jacobian is also growing faster than quadratically in $m$ (in contrast to Theorem~\ref{thm:jac_pj}). These effects can be expected because in the current version of the software \cite{Markovsky.Usevich13JCAM-Software} (7 March 2013) products by $\mblk{\cmat}{i}{j}$ matrices are implemented as products by precomputed matrices, and not as element-wise products.

The examples in this section are reproducible and can be found in the directory 
\begin{center}
{\tt /test\_c/test\_speed } 
\end{center}
of the publicly available software package \cite{Markovsky.Usevich13JCAM-Software}. The results in this section were  obtained on a 2.6GHz Intel Core i5 CPU  with 4GB RAM, running under 64-bit  Linux Mint Debian Edition. 

\section{Conditioning of $\Gamma$, solvability conditions, and accuracy of computations}\label{sec:gamma}
It is important to know the condition number $\kappa_2(\Gamma)$ in order to use the methods of the paper. Indeed, $\kappa_2(\Gamma)$ should be finite in order to apply the efficient algorithms of the paper ($\Gamma$ should be invertible). If $\Gamma$ is invertible, then $\kappa_2(\Gamma)$ determines the accuracy of the computational process. Next, we discuss the invertibility of $\Gamma$ and the behavior of $\kappa_2(\Gamma)$ for mosaic Hankel matrices.

\subsection{Invertibility of $\Gamma$}
In this subsection, we investigate the conditions of invertibility of $\Gamma$ for mosaic Hankel matrices \eqref{eq:mosaic}. By Lemma~\ref{lem:striped} it is necessary that $G(R)$ is of full row rank for each $\HH_{\bfm,n_k}$. By simple calculation of dimensions we have that the necessary condition \eqref{eq:NC} is equivalent to
\[
\sum\limits_{l=1}^q (m_l + n_k -1) \ge n_k d \iff (q-d)n_k \ge q-m,
\]
which is satisfied if $d \le q$ (and is equivalent to  $d\le q$ if $m < n_k-q$ for all $k$).

As it was noted in \cite{Markovsky.Usevich13JCAM-Software}, the condition $d \le q$ can always be satisfied in applications of mosaic Hankel \eqref{eq:prob_slra} to approximate realization, system identification and model reduction. In particular, any Hankel \eqref{eq:prob_slra} can be reduced to an equivalent problem with $d =1$.
\begin{proposition}\label{prop:slra_hankel}
For the Hankel structure $\HH_{m,\np-m+1}$, $r < \min(m,\np-m+1)$ and $\sdat{p} \in \bbR^{\np}$, problem \eqref{eq:prob_slra} is equivalent to \eqref{eq:prob_slra} for the structure $\HH_{r+1, \np-r}$, rank $r$ and  $\sdat{p}$.
\end{proposition}
\begin{proof}
It is sufficient to show that if $r < \min(m,\np-m+1)$
\[
\rrank \HH_{m,\np-m+1} (p) \le r  \iff \rrank \HH_{r+1, \np-r} (p)  \le r,
\]
which is a corollary of \cite[Prop. 5.4]{Heinig.Rost84-Algebraic}.
\end{proof}

The next proposition shows that for the general mosaic Hankel structure, if $d \le q$ then $\Gamma$ is nonsingular for almost all $R$ (is singular on a submanifold of smaller dimension).
\begin{proposition}\label{prop:mosaic_nonsing}
Let  $d \le q$, and $\wmat$ be positive definite. 
\begin{itemize}
\item If $d=1$ then $\Gamma(R)$ is nonsingular for all $R \in \bbR^{1 \times m} \setminus \{0_{1\times m}\}$.
\item If $d > 1$, there exists a permutation matrix $\Pi \in \bbR^{m \times m}$ such that \eqref{eq:GammaG} is nonsingular at $R =\bmx X & \eyemat{d} \emx\Pi$  for any $X \in \bbR^{d\times (m-d)}$.
\end{itemize}
\end{proposition}
\proofapp

Proposition~\ref{prop:mosaic_nonsing} is an extended to the mosaic Hankel matrices Corollary 4.11 from \cite{Markovsky.etal06-Exact}.

\subsection{Accuracy of the computational process}

The key computational procedure in evaluation of the cost function and its derivatives is the solution of a system of equations  $\Gamma u = v$ using Cholesky factorization. The accuracy of this step depends on the condition number of $\Gamma$ \cite{Golub.VanLoan96-Matrix}. In what follows, we review some results concerning conditioning of symmetric block-Toeplitz $\mu$-block-banded $\Gamma$ (the case of layered Hankel structures with block-wise weights, see Theorem~\ref{thm:main}). For convenience, we denote by $\Gamma^{(n)}$ the $nd \times nd$ $\Gamma$ matrix.

Since $\Gamma^{(n)}$ is symmetric positive-semidefinite, $\kappa_2(\Gamma^{(n)})$ (the condition number for $2$-norm) is equal to the  ratio of its extreme eigenvalues $\lambda_{min} (\Gamma^{(n)})$ and $\lambda_{max} (\Gamma^{(n)})$. The properties of the eigenvalues of the matrices $\Gamma^{(n)}$  are determined by their generating function $\rmF: \bbC \to \bbC^{m \times m}$, defined as 
\[
\rmF(z) = \sum\limits_{k=-\mu}^{\mu} \Gamma_k z^{k}.
\]
Since $\rmF(z)$ is Hermitian for all $z$ and $\rmF$ is continuous on the unit circle $\bbT$, we can define
\[
a_\rmF \defeq \min_{\bbT} \lambda_{min} \left(\rmF(z)\right) \ge 0, \quad
b_\rmF \defeq \max_{\bbT} \lambda_{max} \left(\rmF(z)\right) < \infty.
\]
The results of \cite{Miranda.Tilli00SJMAA-Asymptotic} imply that 
$\lambda_{min} (\Gamma^{(n)}) \searrow a_{\rmF} \quad \mbox{and} \quad
\lambda_{max} (\Gamma^{(n)}) \nearrow b_{\rmF}$,
\ie~the eigenvalues are in $[a_\rmF;b_\rmF]$ and converge to the endpoints as $n \rightarrow \infty$. Therefore, 
\[
\kappa_2(\Gamma^{(n)}) \nearrow \frac{b_{\rmF}}{a_\rmF}.
\]
If $\rmF(z)$ is positive definite on $\bbT$, then $\kappa_{2}(\Gamma^{(n)}) \le {b_{\rmF}}/{a_\rmF} < \infty$. Otherwise, $\kappa_{2}(\Gamma^{(n)}) \rightarrow \infty$.

For simplicity, in what follows we consider the case $d=1$. For Hankel structures we have
\begin{equation}
\rmF(z) = R(z) \overline{R(z)} = |R(z)|^2 \quad\mbox{for} \quad z \in \bbT,
\label{eq:Fhankel}\tag{$\rmF_{\HH_{m,n}}$}
\end{equation}
where $R(z) = R_{1,1} +  R_{1,2} z + \ldots + R_{1,m} z^{m-1}$ is the characteristic polynomial of the difference equation defined by $R \in \bbR^{1\times m}$. By Lemma~\ref{lem:layered}, in the general case of mosaic Hankel matrices with block-wise weights $w_l$, the function $\rmF$ is expressed as
\begin{equation}
\rmF(z) = \sum\limits_{k=1}^{q} w_l^{-1} |R^{(k)}(z)|^2 \quad\mbox{for} \quad z \in \bbT.
\label{eq:Fmosaic}\tag{$\rmF_{\HH_{\bfm,n}}$}
\end{equation}
For Hankel structures, the condition number of $\Gamma^{(n)}$ tends to infinity if and only if $R(z)$ has roots on $\bbT$. The growth rate is $n^{2\alpha}$ \cite{Serra98LAA-extreme}, where $\alpha$ is the maximal root multiplicity on $\bbT$. In Fig~\ref{fig:cond}, we plot the dependence of $\kappa_2$ on $n\in 100:1000$ for $K = 50$ random instances of 
\[
R(z) = \prod\limits_{k=1}^4 (z - \lambda_k)(z - \overline{\lambda_k}).
\]
The roots $\lambda_k$ are generated uniformly on the upper half of the unit circle, and the condition number was computed by the function {\tt cond} of Matlab.
\begin{figure}[!ht]
\begin{center}
\includegraphics[height=2cm]{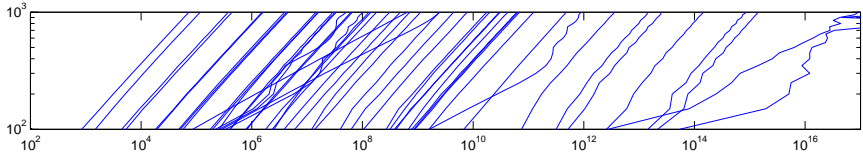}    
\caption{$n$ ($y$-axis) versus $\kappa_2(\Gamma^{(n)})$ ($x$-axis), in log-log scale.} 
\label{fig:cond}                                 
\end{center}                                
\end{figure}
Fig.~\ref{fig:cond} shows that the matrices can be very ill-conditioned. In most cases, we have
$\kappa_n$ grows as $n^2$. In a few cases, the condition number growth is similar to $n^4$, which can be explained by the neighboring roots behaving as multiple roots.

In the mosaic Hankel case, $\kappa_2(\Gamma^{(n)})$ is unbounded if and only if \eqref{eq:Fmosaic} has zeros on $\bbT$.
If only one block has a nonzero weight, then the function \eqref{eq:Fmosaic} is of form \eqref{eq:Fhankel} and $\Gamma^{(n)}$ has the same conditioning properties as that for the scalar Hankel structure. Otherwise, if at least two blocks have weights $w_l < \infty$ (are not fixed), then all corresponding polynomials in \eqref{eq:Fmosaic} need to have common zeros on $\bbT$ in order  for $\rmF$ to have a zero on $\bbT$. 

In system identification---one of the main applications of \eqref{eq:prob_slra} \cite{Markovsky08A-Structured}---the former case corresponds to output-error identification. The latter case corresponds to errors-in-variables system identification. For several polynomials, it is less common (compared to a single polynomial) to have approximately common zeros on or close to the unit circle. Therefore, errors-in-variables identification with \eqref{eq:prob_slra} is typically better conditioned than output-error identification.

\section{Conclusions}
In this paper we considered the structured low-rank approximation problem for general affine structures, weighted 2-norms and fixed values constraints. We used the variable projection principle, which has many advantages when applied to \eqref{eq:prob_slra}. The $\Gamma$ matrix in \eqref{eq:f_unweighted} is structured and its structure is determined by the original matrix structure. For the mosaic Hankel structure \eqref{eq:mosaic}, the $\Gamma$ matrix is block-banded/block-Toeplitz, depending on the structure of the weight matrix. This allows us to evaluate the cost function $f(R)$ and its derivatives in $O(m^2 n)$ flops ($O(mn)$ for the cost function and the gradient if $\Gamma$ is block-Toeplitz). The approach in this paper can be applied to other matrix structures $\struct$, where the structure of $\Gamma$ (\eg~sparseness) can be exploited for efficient computations.

Whenever possible, we considered the most general cases (structures, weights) and developed the algorithms for the general affine structure. We showed how the cost functions for the blocked structures (layered, striped) can be expressed through the cost functions of the blocks. This allowed us to reduce the mosaic Hankel case to the scalar Hankel case and helped to simplify the derivations of the algorithms and their complexities (compared to \cite{Markovsky.etal04NLAA-computation,Markovsky.etal06-Exact}). We also considered fixed values as a special case of the weighted norm, which also simplified  treatment of this case. 

The developed algorithms require invertibility of the $\Gamma$ matrix. For 
affine structures, a necessary condition of invertibility is $\np \ge nd$, which for mosaic Hankel matrices is satisfied if the rank reduction $d$ does not exceed the number of block rows $q$ in \eqref{eq:mosaic}. If the condition $d \le q$ is satisfied, the $\Gamma$ matrix is nonsingular for a generic $R$. 

The accuracy of the computations mainly depends on the conditioning of $\Gamma$. If $\Gamma$ is block-Toeplitz, the properties of its condition number can be expressed through its generating function. Deeper investigation of properties of $\Gamma$, such as invertibility and conditioning, is a direction of future research.

\section*{Acknowledgements}
We would like to thank anonymous reviewers for the very valuable comments. The research leading to these results has received funding from the European Research Council under the European Union's Seventh Framework Programme (FP7/2007-2013) / ERC Grant agreement no. 258581 ``Structured low-rank approximation: Theory, algorithms, and applications''.

\appendix
\section{Proofs}\label{sec:proofs}

\begin{proof}[Proof of Proposition~\ref{prop:weighted_cases}]
If $\wmat$ is positive definite, then \eqref{eq:repar_w} is an invertible transformation and problems \eqref{eq:wslra_reduced} and \eqref{eq:prob_slra} are equivalent by \eqref{eq:SdeqS} and \eqref{eq:wnorm_2norm}.

Now consider the case when $\wmat$ is given by \eqref{eq:wtoW}. We have that
\[
\rchol{\wmat}^{-1} = \diag(\sqrt{w_1}^{-1}, \ldots, \sqrt{w_{\np}}^{-1}),
\]
where $\sqrt{\infty}^{-1} = 0$. Therefore, the parametrization~\eqref{eq:repar_w} keeps $\sest{p}_k = (\sdat{p})_k$ for fixed values ($k$ such that $w_k = \infty$) and runs over all possible $\sest{p}_k$ for other $k$. Therefore, all admissible $\sest{p}$ for \eqref{eq:prob_slra} (with fixed values constraints) can be parametrized as \eqref{eq:repar_w}.

Instead of \eqref{eq:wnorm_2norm}, in the case \eqref{eq:wtoW} we have that
\begin{equation}
\|\sdelt{p}\|^2_{2} =
\|\sest{p} - \sdat{p}\|^{2}_{\wmat} + \sum\limits_{\{k:\ w_k = \infty\}} (\sdelt{p_k})^2,
\label{eq:deltp_2norm_fixed}
\end{equation}
and by \eqref{eq:SdeqS} and \eqref{eq:deltp_2norm_fixed}, \eqref{eq:wslra_reduced} is equivalent to 
\begin{equation}
\minim_{\sdelt{p} \in \bbR^{\np}} \; \mbox{(\ref{eq:deltp_2norm_fixed})} \; \sto \;
\rrank \structdelt(\sdelt{p}) \le r.
\label{eq:wslra_reduced_fixed}
\end{equation}
It is straightforward  that $\structdelt(\sdelt{p})$ does not depend on  $\sdelt{p}_k$ corresponding to the fixed values. From the first-order optimality conditions of \eqref{eq:wslra_reduced_fixed}, all local minima of \eqref{eq:wslra_reduced_fixed} must satisfy $\sdelt{p_k} = 0$ for $w_k = \infty$. Hence,
the local minima of \eqref{eq:prob_slra} coincide with the local minima of \eqref{eq:wslra_reduced_fixed}.
\end{proof}

\begin{proof}[Proof of Proposition~\ref{prop:gradient}]
First, note that if the differential of $\costfun$ is represented as
\begin{equation*}
df(R, H) = \trace (A H^\top),
\end{equation*}
where $\trace$ is the trace operator, then $\mgrad{d}{m}(\costfun) = A$. From \eqref{eq:invgresd_resd}, the differential of $\costfun$ is given by
\begin{equation}
df(R, H) =  2 \invgresd^\top d\resds(R,H) - \invgresd^\top d\Gamma (R,H) \invgresd.
\label{cost_diff1}\tag{$df(R, H)$}
\end{equation}
By \eqref{eq:resd} and \eqref{eq:invgresdmat}, the first term in \eqref{cost_diff1} is equal to 
\[
2\invgresd^\top d\resds(R,H) = \trace \big(2\invgresdmat (H \struct(\sdat{p}))^\top\big) =  
\trace \big(2 \invgresdmat \struct^{\top}(\sdat{p}) H^\top\big).
\]
By \eqref{eq:gamma_block}, the second term in \eqref{cost_diff1} is equal to
\begin{align*}
\invgresd^\top d\Gamma (R,H) \invgresd  
&= \sum\limits_{i,j=1}^{n} \invgresdbl{i}^\top d \mblk{\Gamma}{i}{j} (R,H) \invgresdbl{j}\\
&= \sum\limits_{i,j=1}^{n}
\invgresdbl{i}^\top (H \mblk{\cmat}{i}{j} R^{\top} + R \mblk{\cmat}{i}{j} H^{\top}) \invgresdbl{j} \\
&= \sum\limits_{i,j=1}^{n}
 \trace \Big( \big( \invgresdbl{i} \invgresdbl{j}^{\top} R \mblk{\cmat}{i}{j}^\top + 
\invgresdbl{j} \invgresdbl{i}^{\top} R \mblk{\cmat}{i}{j} \big) H^\top \Big),
\end{align*}
which in combination with $\mblk{\cmat}{i}{j} = \mblk{\cmat}{j}{i}^{\top}$ yields \eqref{eq:mgrad_f}.
\end{proof}

\begin{proof}[Proof of Proposition~\ref{prop:jacobian}]
The equation \eqref{eq:jacobian} can be obtained by differentiation of 
 \eqref{eq:ker_slra_cos2t}.
The first summand in \eqref{eq:zij} can be expressed as 
\[
\frac{\partial \resds}{\partial R_{ij}} = 
\mvec \left(\ortvec{i} (\ortvec{j})^{\top} \struct(\sdat{p})\right) = (\struct^{\top}(\sdat{p}))_{:,j} \otimes \ortvec{i}. 
\]
The second summand is
\[
\frac{\partial \Gamma}{\partial R_{ij}} \invgresd = \vcol(a_1, \ldots, a_n),\quad a_l \in \bbR^{d},
\]
where
\[
\begin{split}
a_l &= \sum_{k=1}^{n} \frac{\partial \Gamma_{l,k}}{\partial R_{ij}} \invgresdbl{k} 
= \sum_{k=1}^{n} \left(\ortvec{i} (\ortvec{j})^{\top} \mblk{\cmat}{l}{k} R^{\top} +   R \mblk{\cmat}{l}{k} \ortvec{j} (\ortvec{i})^{\top}\right)  \invgresdbl{k}  \\
&= \left(\sum_{k=1}^{n} (\ortvec{j})^{\top} \mblk{\cmat}{l}{k} R^{\top} \invgresdbl{k} \right) \ortvec{i}  +  \sum_{k=1}^{n} \invgresdmat_{i,k} R \mblk{\cmat}{l}{k} \ortvec{j}. \quad\qedhere
\end{split}
\]
\end{proof}

\begin{proof}[Proof of Theorem~\ref{thm:cost_grad}]
The complexities for each individual step of Algorithm~\ref{alg:cost_fun_cholesky} are given in Table~\ref{tab:cost_fun_cholesky}. The complexities of the steps for Algorithm~\ref{alg:gradient} are given in Table~\ref{tab:gradient}.
\begin{table}[htb!]
\centering
\begin{tabular}{p{4cm} p{0.5cm} p{7.5cm}}
1. Compute \eqref{eq:resd} &---&  $O(dmn)$, since it is a multiplication of a $d \times m$ matrix $R$ by $m \times n$ matrix $\struct(p)$\\
2. Compute \eqref{eq:gamma_block} &---& $O(d^2 \bw m n)$. We need to compute $\bw n$ matrices $\mblk{\Gamma}{i}{j}$. By Theorem~\ref{thm:main}, each $\mblk{\cmat}{i}{j}$ has $\le m$ nonzero elements, and each $\mblk{\Gamma}{i}{j}$ needs  $d^2 m$ multiplications. \\
3. Compute  $\rchol{\Gamma}$ &---& $O(d^3 \bw^2 n)$, by Lemma~\ref{lem:comp_banded_cholesky}. \\
4. Solve $\rchol{\Gamma}^{\top}  \lschol(R) = \resd{R}$ &---& $O(d^2 \bw n)$, by Lemma~\ref{lem:comp_banded_cholesky}. \\
5. Compute $\|\lschol(R)\|_2^2$ &---& $O(nd)$.
\end{tabular}
\caption{Complexity of the steps of Algorithm~\ref{alg:cost_fun_cholesky}}\label{tab:cost_fun_cholesky}
\end{table}
\begin{table}[htb!]
\centering
\begin{tabular}{p{4cm} p{0.5cm} p{7.5cm}}
1. Steps 1--3 of Algorithm~\ref{alg:cost_fun} &---&  $O(d^3 \bw^2 n)$, see Table~\ref{tab:cost_fun_cholesky}.\\
2. Compute $2\invgresdmat \struct^\top(\sdat{p})$ &---& $O(nmd)$, since it is multiplication of a $d \times n$ matrix by $n \times m$ matrix.\\
3. Compute the second term of \eqref{eq:mgrad_f} &---& $O(d \bw m n)$. We need to compute $2\bw n$ products 
$\invgresdbl{j} \invgresdbl{i}^\top R \mblk{\cmat}{i}{j}$. By
Theorem~\ref{thm:main}, $\mblk{\cmat}{i}{j}$ has only one nonzero diagonal and $u :=\invgresdbl{i}^\top R \mblk{\cmat}{i}{j}$
can be computed in $dm$ flops. The  product $\invgresdbl{j} u$  takes $dm$ flops. 
\end{tabular}
\caption{Complexity of the steps of Algorithm~\ref{alg:gradient}}\label{tab:gradient}
\end{table}
\end{proof}

\begin{proof}[Proof of Theorem~\ref{thm:jac_pj}]
For each \eqref{eq:zij1}, due to block-bandedness of $\cmat$, we need to compute  $2\bw n$ products of the form 
$a_{j,k,l}^{\top} \invgresdbl{k}$, where
\[
a_{j,k,l}^{\top} \defeq (\ortvec{j})^{\top} \mblk{\cmat}{l}{k} R^{\top}. 
\]
By Theorem~\ref{thm:main}, each $(\ortvec{j})^{\top}  \mblk{\cmat}{l}{k}$ has at most one nonzero element, and computing $a_{j,k,l}^{\top}$ takes $d$ multiplications. Another $d$ multiplications are needed to compute the inner product of $a_{j,k,l}$ and~$\invgresdbl{k}$.
The computation of \eqref{eq:zij1} is repeated $m$ times, which leads to $O(d \bw m n)$ complexity.

For each \eqref{eq:zij2}, we need to compute $2\bw n$ products of the form $\invgresdmat_{i,k} R \mblk{\cmat}{1}{k} \ortvec{j}$, where each product has complexity $d$,  as in the previous step. The computation of \eqref{eq:zij2} is repeated $md$ times, which leads to $O(d^2 \bw mn)$ complexity.

For pseudo-Jacobian, we need  to solve $md$ times the banded system with the Cholesky factor $L^{\top}_R$ and $z_{ij}$, which has complexity $O(d^3 \bw m n)$. For Jacobian, we also need to multiply each $\Gamma^{-1} z_{ij}$ from the left by $G^{\top} (R)$, which has complexity $dmn$ by Note~\ref{not:corr_comp}. Therefore, this step has additional complexity $O(d^2 m^2 n)$.
\end{proof}

\begin{proof}[Proof of Proposition~\ref{prop:toeplitz_gradient}]
The second term in \eqref{eq:mgrad_f} can be represented as
\[
\begin{split}
&\sum\limits_{k=-\bw}^\bw \sum\limits_{i=\max(1-k,1)}^{\min(n-k,n)}
\invgresdbl{i+k} \invgresdbl{i}^\top R \cmat_{k} = \sum\limits_{k=0}^\bw \sum\limits_{i=1}^{n-k}
\invgresdbl{i+k} \invgresdbl{i}^\top R \cmat_{k} +
\sum\limits_{k=1}^\bw \sum\limits_{i=1}^{n-k}
\invgresdbl{i} \invgresdbl{i+k}^\top R \cmat_{k}^{\top}.
\end{split}
\]
It is easy to see that for $k \ge 0$
\[
\sum\limits_{i=1}^{n-k} \invgresdbl{i+k} \invgresdbl{i}^\top = 
\invgresdmat (\shiftmat{n}^\top)^{k} (\invgresdmat)^\top = N_k.
\]
\end{proof}

\begin{proof}[Proof of Theorem~\ref{thm:comp_blt}]
Since only $\bw$ matrices $\Gamma_k$ need to be computed, step 2  of Algorithm~\ref{alg:cost_fun} can be performed in $d^2 \bw m$ flops. For Cholesky factorization (Step 3) the algorithms exploiting the Toeplitz structure \cite{VanHuffel.etal04ICSM-High} can be used. Their  complexity  is $O(d^3 \bw n)$. 

For computation of the gradient, we use Proposition~\ref{prop:toeplitz_gradient}. Each $N_k$ can be computed in $d^2 n$ multiplications. The computation of $N_k R \cmat_k$ has complexity $O(d^2 m)$ due to sizes of the involved matrices. Therefore the total  complexity of Step 3 of Algorithm~\ref{alg:gradient} if $O(d^3 \bw n)$.
\end{proof}

\begin{proof}[Proof of Proposition~\ref{prop:mosaic_nonsing}]
By Lemma~\ref{lem:striped}, we consider only the case of mosaic Hankel matrix $\HH_{\bfm,n}$. Since $\wmat$ is a positive definite matrix, by \eqref{eq:lin_matr2} it is sufficient to prove that $(\eyemat{n} \otimes R) \maffine{\struct}$ is of full row rank, therefore we may assume that $\cmat = \eyemat{\np}$, without loss of generality. We can also limit our consideration to the case when all $m_k$ are equal, by noting that
\[
\rrank G_{\HH_{\bfm,n}} (R) =  \rrank G_{\HH_{\bfm',n}} (R'),
\]
where $\bfm' \defeq \bmx m_{max} & \cdots & m_{max} \emx$, $m_{max} \defeq \max m_l$, and
\[
R' \defeq \bmx 0_{d\times(m_{max} - m_1)} & R^{(1)} & \cdots & 0_{d\times(m_{max} - m_q)} & R^{(q)} \emx \in \bbR^{d \times q m_{max}},
\]
where $R^{(k)} \in \bbR^{d\times m_k}$ are defined in \eqref{eq:Rpart}.

Now we assume that $\bfm  = \bmx m_1 & \cdots & m_1 \emx$. Let us consider $R = Q \Pi^{(m_1,q)}$, where $\Pi^{(m_1,q)}$ is defined in \eqref{eq:perm}. Then by \eqref{G_h}, we have that
\[
G_{\HH_{\bfm,n}} (R) \Pi^{(m_1+n-1, q)} = 
\bmx
Q^{(1)}    &  \ldots    & Q^{(m_1)}  &          & \\
           & \ddots     &            &  \ddots  & \\
           &            & Q^{(1)}    & \ldots   & Q^{(m_1)} 
\emx \in \bbR^{nd \times \np}.
\]
where $Q = \bmx Q^{(1)}    & Q^{(2)} & \ldots  & Q^{(m_1)} \emx$ is the partition of $Q$ into $Q^{(k)} \in \bbR^{d\times q}$.

If $d=1$, then the matrix $G_{\HH_{\bfm,n}} (R) \Pi^{(m_1+n-1, q)}$ is of full row rank for all nonzero $Q$. For $d \ge 1$, consider $Q = \bmx X & \eyemat{d} \emx$. Since $d \le q$, $Q^{(m_1)}$ has the form $Q^{(m_1)} = \bmx *& \eyemat{d} \emx$ and $G_{\HH_{\bfm,n}} (R) \Pi^{(m_1+n-1, q)}$ is of full row rank due to its row echelon form. 
\end{proof}

\bibliographystyle{model1-num-names}
\bibliography{sla,lsnls,slra,grassopt,textbooks,kdu}

\end{document}

%% file: comsymb.tex
\def\calI{\mathcal{I}}
\def\calJ{\mathcal{J}}

\def\bfm{\mathbf{m}}
\def\bfn{\mathbf{n}}

\def\bbC{\mathbb{C}}

\def\bbN{\mathbb{N}}

\def\bbR{\mathbb{R}}

\def\bbT{\mathbb{T}}

\def\rmF{\mathrm{F}}

%% file: comdefs.tex
\newcommand{\defeq}{\mathrel{:=}}
\newcommand{\mvec}{\mathop{\text{vec}}}
\newcommand{\rrank}{\mathop{\text{rank}}}
\newcommand{\rowspan}{\mathop{\text{rowspan}}}
\newcommand{\minim}{\mathop{\text{\;minimize\;}}}
\newcommand{\sto}{\mathop{\text{\;subject\ to\;}}}
\newcommand{\diag}{\mathop{\mathrm{diag}}} 
\newcommand{\trace}{\mathop{\text{tr}}} 
\newcommand{\vcol}{\mathop{\mathrm{col}}} 
\def\blkdiag{\mathop{\mathrm{blkdiag}}} 

\newcommand{\bmx}{\begin{bmatrix}}
\newcommand{\emx}{\end{bmatrix}}
\newcommand{\bsm}{\left[\begin{smallmatrix}}
\newcommand{\esm}{\end{smallmatrix}\right]}

\newcommand{\ie}{\emph{i.e.}}
\newcommand{\eg}{\emph{e.g.}}

\def\rowspan{\mathop{\mathrm{rowspan}}}

\def\rowvec#1#2{[\,#1 \; \cdots \; #2\,]}

%% file: slraalgdefs.tex
\newcommand{\resds}{s} 
\def\resd#1{\resds(#1)} 
\newcommand{\invgresd}{y} 
\newcommand{\invgresdmat}{Y} 
\def\invgresdbl#1{\invgresdmat_{{:,#1}}} 
\newcommand{\HH}{\mathscr{H}}   
\def\struct{\mathscr{S}} 
\newcommand{\bw}{\mu}   
\def\ortvec#1{{e^{(#1)}}}   
\newcommand{\lschol}{g_\resds}
\newcommand{\pseudopartial}{\partial^{(\mathfrak{p})}}
\newcommand{\pseudoz}{z^{(\mathfrak{p})}}
\def\mgrad#1#2{\mathop{\nabla_{{#1}\times {#2}}}} 

\def\rchol#1{\mathrm{L}_{#1}}

\def\costfun{f}
\def\structdelt{\struct_{\Delta}}
\def\np{n_\mathsf{p}}
\def\maffine#1{\mathbf{S}_{#1}} 

\def\wmat{\mathrm{W}}
\def\cmat{\mathrm{V}}

\def\intrng#1#2{{{#1}:{#2}}}

\def\sdat#1{{#1}_{ \mathtt{D}}}
\def\sest#1{\widehat{#1}}
\def\sdelt#1{{\Delta #1}}
\def\vfix#1{{\text{fix}}(w)}
\def\shiftmat#1{\mathrm{J}_{#1}}
\def\eyemat#1{\mathrm{I}_{#1}}

\def\mblk#1#2#3{#1_{\#{#2}{#3}}}

%% file: engthm.tex
\newtheorem{algorithm}{Algorithm} 
\newtheorem{theorem}{Theorem} 
\newtheorem{problem}{Problem} 
\newtheorem{corollary}{Corollary} 
\newtheorem{lemma}{Lemma} 
\newtheorem{proposition}{Proposition} 
\newtheorem{note}{Note} 
 
\newtheorem{example}{Example}